\documentclass[12pt,psamsfonts]{amsart}

\usepackage{amssymb,amsfonts,amsmath}
\usepackage{enumerate}
\usepackage{mathrsfs}
\usepackage{fullpage}
\usepackage{xspace}
\usepackage[margin=1.0in]{geometry}
\usepackage{tcolorbox}
\usepackage{tikz-cd}
\usepackage{color}
\usepackage{aliascnt}
\usepackage[foot]{amsaddr}
\usepackage{hyperref}
\usepackage{graphicx}
\usepackage[algoruled,linesnumbered]{algorithm2e}

\usepackage{enumitem}
\setlist[enumerate]{label=$(\mathrm{\arabic*})$, leftmargin=*}
\setlist[itemize]{leftmargin=*}

\newtheorem{thm}{Theorem}[section]

\newaliascnt{theo}{thm}
\newtheorem{theo}[theo]{Theorem}
\aliascntresetthe{theo}

\newaliascnt{cor}{thm}
\newtheorem{cor}[cor]{Corollary}
\aliascntresetthe{cor}

\newaliascnt{prop}{thm}
\newtheorem{prop}[prop]{Proposition}
\aliascntresetthe{prop}

\newaliascnt{lem}{thm}
\newtheorem{lem}[lem]{Lemma}
\aliascntresetthe{lem}

\newaliascnt{conj}{thm}
\newtheorem{conj}[conj]{Conjecture}
\aliascntresetthe{conj}

\newaliascnt{que}{thm}

\aliascntresetthe{que}

\newaliascnt{ass}{thm}

\aliascntresetthe{ass}

\newaliascnt{defnot}{thm}

\aliascntresetthe{defnot}

\newaliascnt{princ}{thm}

\aliascntresetthe{princ}

\theoremstyle{remark}
\newaliascnt{rem}{thm}
\newtheorem{rem}[rem]{Remark}
\aliascntresetthe{rem}

\theoremstyle{definition}

\newaliascnt{defn}{thm}
\newtheorem{defn}[defn]{Definition}
\aliascntresetthe{defn}

\newaliascnt{exmp}{thm}

\aliascntresetthe{exmp}

\newaliascnt{notn}{thm}
\newtheorem{notn}[notn]{Notation}
\aliascntresetthe{notn}

\newtheorem{rems}[thm]{Remarks}

\newcommand{\Z}{\mathbb{Z}\xspace}
\newcommand{\QQ}{\overline{\mathbb{Q}}\xspace}
\newcommand{\C}{\mathbb{C}\xspace}
\newcommand{\F}{\mathbb{F}\xspace}
\newcommand{\Q}{\mathbb{Q}\xspace}

\newcommand{\kk}{\overline{k}\xspace}

\DeclareMathOperator{\ord}{ord}
\DeclareMathOperator{\Sym}{Sym}

\DeclareMathOperator{\alb}{alb}

\DeclareMathOperator{\Pic}{Pic}

\DeclareMathOperator{\Jac}{Jac}

\DeclareMathOperator{\Alb}{Alb}

\DeclareMathOperator{\CH}{CH}
\DeclareMathOperator{\rk}{rank}

\DeclareMathOperator{\im}{im}

\DeclareMathOperator{\Kum}{Kum}

\makeatletter
\let\c@equation\c@thm
\makeatother
\numberwithin{equation}{section}

\newcommand{\red}{\color{red}}

\newcommand{\blue}{\color{blue}}
\newcommand{\black}{\color{black}}

\newcommand{\new}{_{\mathrm{comp}}}

\title{Hyperelliptic curves mapping to abelian varieties and applications to Beilinson's conjecture for zero-cycles}

\author[*]{Evangelia Gazaki*} \address[*]{\normalfont Department of Mathematics, University of Virginia, 221 Kerchof Hall, 141 Cabell Dr., Charlottesville, VA, 22904, USA. Email: \texttt{eg4va@virginia.edu}}
\author[**]{Jonathan Love**} \address[**]{\normalfont Mathematical Institute, Leiden University, Einsteinweg 55, 2333 CC Leiden, the Netherlands. Email: \texttt{j.r.love@math.leidenuniv.nl}}
\begin{document}

	\begin{abstract} 
		Let $A$ be an abelian surface over an algebraically closed field $\kk$ with an embedding $\kk\hookrightarrow\C$. When $A$ is isogenous to a product of elliptic curves, we describe a large collection of pairwise non-isomorphic hyperelliptic curves mapping birationally into $A$. For infinitely many integers $g\geq 2$, this collection has infinitely many curves of genus $g$, and no two curves in the collection have the same image under any isogeny from $A$.  Using these hyperelliptic curves, we find many rational equivalences in the Chow group of zero-cycles $\CH_0(A)$.  We 
		use these results to give some progress towards 
		Beilinson's conjecture for zero-cycles, which predicts that for a smooth projective variety $X$ over $\QQ$ the kernel of the Albanese map of $X$ is zero.

	\end{abstract}

	\maketitle
	
	\bigskip

\section{Introduction}

 In this article,  unless otherwise specified, we will be working over an algebraically closed field $\kk$. 
For a smooth projective variety $X$ over $\kk$ we consider the Chow group $\CH_0(X)$ of zero-cycles modulo rational equivalence on $X$. This group has a filtration 
\[\CH_0(X)\supset F^1(X)\supset F^2(X)\supset 0,\] where $F^1(X)=\ker\left(\CH_0(X)\xrightarrow{\deg}\Z\right)$ is the kernel of the degree map, sending the class $[x]$ of a closed point $x\in X$ to $1$, and $F^2(X)=\ker\left(F^1(X)\xrightarrow{\alb_X}\Alb_X(\kk)\right)$ is the kernel of the Albanese map of $X$.  
When $\kk=\C$ and the variety $X$ has positive geometric genus $p_g(X)$, the Albanese kernel $F^2(X)$ is known to be enormous (see \cite{Mumford1968, Bloch1975}) and it cannot be parametrized by  an algebraic variety. More generally, if $X$ is defined over a large transcendental field $\kk$ and $p_g(X)>0$, then $F^2(X)\neq 0$. 
\black 
  In contrast, the situation is conjectured to be vastly different over $\QQ$. A famous conjecture 
of Beilinson predicts the following. 
\begin{conj}\label{beilconj} (Beilinson \cite{Beilinson1984}) Let $X$ be a smooth projective variety over $\QQ$. Then the Albanese map is injective. That is, $F^2(X)=0$. 
\end{conj}
 To a large extent, this conjecture is still  very far out of reach. In fact, 
to the authors' knowledge there is not a single example of a smooth projective surface $X$ with $p_g(X)>0$ that is known to satisfy this. 
The purpose of this article is to give some hope towards \autoref{beilconj} for two special classes of surfaces with positive geometric genus, namely abelian surfaces  and their associated Kummer surfaces. 

Kummer surfaces form a special class of projective $K3$ surfaces. For $X$ a general $K3$ surface, we have $\Alb_X=0$. Thus, proving Beilinson's conjecture for $X/\QQ$ amounts to showing that every two $\QQ$-points $x,y$ are rationally equivalent. Note that this is very far from being true for two general $\C$-points. A celebrated result of Beauville and Voisin (\cite{Beauville/Voisin2004}) states that if $X/\kk$ is a $K3$ surface and $x,y\in X(\kk)$ are points that lie on some (possibly different) rational curve inside $X$, then $x,y$ are rationally equivalent. Thus, a wishful hope would be that $X(\QQ)$ can be covered by rational curves. In fact, this is predicted by a conjecture of Bogomolov from 1981 (see \cite[p.~2]{bogomolov_tschinkel} for a more recent reference). 

Next, consider an abelian surface $A$ over $\kk$ with zero element $0\in A(\kk)$. In this case $\Alb_A=A$ and the Albanese kernel $F^2(A)$ is generated by zero-cycles of the following form 
\[F^2(A)=\langle z_{a,b}:=[a+b]-[a]-[b]+[0], a,b\in A\rangle.\]  In \autoref{abelsurfacesection} we prove a theorem for abelian surfaces analogous to the Beauville-Voisin result by replacing rational curves with hyperelliptic curves. 
\black 
We say that a point $a\in A(\kk)$ is \textit{hyperelliptic} if some nonzero multiple of $a$ lies in the image of a morphism $f \colon H\to A$, where~$H$ is a hyperelliptic curve over $\kk$ such that  negation on $A$ induces the hyperelliptic involution on $H$ \black (see \autoref{goodpoint}). Our first theorem is the following.

\begin{theo}\label{abel} (cf.~\autoref{abelsurface1}) Let $A$ be an abelian surface over an algebraically closed field $\kk$ and let $a,b\in A(\kk)$. Suppose there exist nonzero integers $m,m'$ such that $a$, $b$, and $ma+m'b$ are hyperelliptic points.  Let $B_{a,b}$ be the divisible hull of the subgroup of $A(\kk)$ generated by $a,b$ (that is, $x\in B_{a,b}$ if there exists some nonzero $\ell \in\Z$ such that $\ell x\in\langle a,b\rangle$).  Then for every $c,d\in B_{a,b}$   the zero-cycle 
\[z_{c,d}:=[c+d]-[c]-[d]+[0]\] vanishes in the Albanese kernel $F^2(A)$.
\end{theo}
 For the purposes of the introduction we chose to state \autoref{abel} in a simplified form; a much more general statement can be found in \autoref{abelsurface1}. 
 We emphasize that the points $a,b,ma+nb$ may lie in the images of morphisms from three distinct hyperelliptic curves. 
  The proof of \autoref{abel} involves Rojtman's theorem (\cite{Rojtman1980}) and two simple steps: (i) if $a\in A$ is a hyperelliptic point, then $z_{a,-a}=0$, and (ii) the zero-cycle $z_{a,b}$ is bilinear on $a,b$.   

The main motivation for looking at hyperelliptic curves mapping to $A$, besides that they give easy relations in $F^2(A)$, is that there is a natural way to construct such curves by pulling back rational curves in the Kummer surface $K=\Kum(A)$ associated to $A$. In particular, if Bogomolov's conjecture is true for the Kummer surface $K$, that is, if every point of $K(\QQ)$ lies on a rational curve, \black then  \autoref{abel} gives that  Beilinson's conjecture will be true for both $K$ and $A$. In \autoref{curvessection}, following this recipe, we construct a large countable collection of hyperelliptic curves mapping to an abelian surface isogenous to a product of elliptic curves.
\black 

 \begin{theo}\label{infinite_H:intro} 
Let $\kk$ be an algebraically closed field admitting an embedding $\kk\hookrightarrow\C$. Let~$A$ be an abelian surface over  $\kk$  and suppose $A$ is isogenous to a product $E\times E'$ of two elliptic curves. 
 For infinitely many positive integers $g\geq 2$, there exist infinitely many non-isomorphic genus $g$ hyperelliptic curves $H$ with maps $f \colon H\to A$ such that 
 negation on $A$ induces the hyperelliptic involution on $H$ \black and such that $(\mu\circ f)(H)$ is birational to $H$ for all 
 isogenies $\mu$ from~$A$.
\end{theo}

 The proof involves considering   the Kummer surface $K=\Kum(E\times E')$ associated to $E\times E'$ and  giving it the structure of an elliptic fibration. Using addition in the Mordell-Weil lattice we produce infinitely many rational curves in $K$.  Each choice of a Mordell-Weil element pulls back to a hyperelliptic curve in $E\times E'$, with genus lying in a certain explicit interval depending on the canonical height of the corresponding Mordell-Weil element. We show that both endpoints of this interval can be made arbitrarily large, which gives us curves~$H$ of arbitrarily large genus. Finally, to obtain infinitely many non-isomorphic hyperelliptic curves of the same genus,  we replace each of $E$ and $E'$ by an isogenous curve \black and then consider the same construction,  
 noting that over an algebraically closed field of characteristic $0$ any isogeny class of elliptic curves has infinitely many isomorphism classes.

The hyperelliptic curves we construct can be made very explicit. In fact, we can algorithmically write down their Weierstrass equations and the morphisms to $A$. See \autoref{example:explicit} for examples with genus $g=2$ and $6$. Heuristically, we expect that for every positive integer $g\equiv 2\mod 4$, there are infinitely many non-isomorphic hyperelliptic curves~$H$ satisfying the conditions of \autoref{infinite_H:intro}  (see \autoref{unexpectedintersection}).

\subsection{Progress towards Beilinson's Conjecture} Coming back to \autoref{beilconj}, 
the advantage of \autoref{abel} is that it requires much less than Bogomolov's conjecture. Indeed, suppose that the abelian surface $A/\QQ$ is defined over an algebraic number field $L$. Proving \autoref{beilconj} amounts to showing that for every finite extension $F/L$,  $z_{a,b}=0$ for all $a,b\in A(F)$. \autoref{abel} together with the Mordell--Weil theorem reduce this question to finding only finitely many hyperelliptic points in $A(F)$, even though the rank of $A(F)$ can be arbitrarily large (see \autoref{beilreduce1}). 

 In the special case when $A$ is isogenous to a product $E\times E'$, much more can be said. First, we only need to produce $\rk(E(F))\cdot\rk(E'(F))$ hyperelliptic points (see \autoref{elliptic2}). Second and more importantly, if we assume that $A[2]\subset A(F)$, then infinitely many of the sections of the elliptic fibration we use in \autoref{infinite_H:intro} are defined over $F$, yielding an infinite collection of hyperelliptic curves $H/F$ mapping to $E\times E'$.

 In \cite{GazakiLove2022}, the authors considered product surfaces $A=E\times E'$ together with a small collection (up to six) of genus $2$ \black hyperelliptic curves $H$ mapping into $A$. For some pairs $E,E'$, hyperelliptic points arising from the image of $H$ were sufficient to verify that $z_{a,b}=0$ for all $a,b\in A(\Q)$. Using a larger collection of hyperelliptic curves as constructed in the proof of \autoref{infinite_H:intro}, we can verify this result for a substantially larger collection of pairs $E,E'$; see \autoref{computations}. It is worth pointing out that even though it is heuristically rare for curves of large genus to have \black rational points, we did manage to find many new examples where hyperelliptic curves of  genus greater than or equal to $6$ produced \black new relations over $\Q$. Moreover, we found examples where we could not find enough relations using only \black the sections of the elliptic fibration of $\Kum(E\times E')$, but  by replacing each of $E$ and $E'$ by an isogenous curve and applying the same construction, enough relations could be produced. \black 
\black


\begin{rems} (1.) Given any two hyperelliptic curves $H_1,H_2$ from \autoref{infinite_H:intro} with corresponding maps $f_1,f_2$, there are no relations of the form $m_1f_1(H_1)=m_2f_2(H_2)$ for any nonzero integers $m_1,m_2$; this is due to the fact that multiplication by $m_i$ is a self-isogeny of~$A$ and so $m_if_i(H_i)$ is birationally equivalent to $H_i$. Due to this, \autoref{infinite_H:intro} gives us a huge untapped source of hyperelliptic points with no obvious relations between them.



 (2.)  We note that even for a simple Jacobian $J$ of a genus $2$ curve there are old constructions (\cite{Humbert, Bost/Mestre}) that can be used to produce countably many genus $2$ (and hence hyperelliptic) curves mapping to $J$ (see \autoref{oldconstructions} for more details). 
 
\end{rems}

\subsection*{Results in higher dimensions} 
In \autoref{abvars} we discuss extensions of \autoref{abel} to abelian varieties of arbitrary dimension. In fact, we prove an analog of \autoref{abel} even when $\dim(A)\geq 3$  with a slightly more involved argument (cf.~\autoref{beilreducehigher}). Given that in dimension $\geq 3$ the corresponding Kummer variety $K=\Kum(A)$ is usually not Calabi-Yau, it might be too optimistic to expect that we have enough hyperelliptic curves. However, this could \black  be true in special cases, for example for Jacobians of hyperelliptic curves (see \cite[Theorem 7.1, Remark 7.7]{bogomolov_tschinkel} for some motivation towards these expectations).

\subsection{Acknowledgement} The first author's research was partially supported by the NSF grants DMS-2001605 and DMS-2302196. The second author was supported by a CRM--ISM postdoctoral fellowship. We are truly grateful to Professor Arnaud Beauville who helped us simplify the proof of \autoref{abel} significantly. We would also like to heartily thank Professors Wayne Raskind, Ari Shnidman, Ravi Vakil and Akshay Venkatesh who showed interest in our work and for useful discussions.  Lastly, we would like to thank the referee whose suggestions improved significantly the paper.  
\vspace{5pt} 

\section{Rational Equivalences on Abelian Surfaces}\label{abelsurfacesection}
\black 
For a smooth projective variety $X$ over an algebraically closed field $\kk$ we consider the Chow group $\CH_0(X)$ of zero-cycles modulo rational equivalence on $X$.

\begin{notn} Following the introduction, we will denote by $F^1(X)$ the kernel of the degree map, $\CH_0(X)\xrightarrow{\deg}\Z$, and by $F^2(X)$ the kernel of the Albanese map, $F^1(X)\xrightarrow{\alb_X}\Alb_X(\kk)$. 
\end{notn}
 The Albanese variety $\Alb_X$ of $X$ is an abelian variety, which is dual to the Picard variety of $X$.  When $X$ is an abelian variety, $\Alb_X=X$, while when $X$ is a $K3$ surface, $\Alb_X=0$. 
 Since $\kk$ is algebraically closed,  Rojtman's theorem  (cf.~\cite{Rojtman1980}) gives that the Albanese map induces an isomorphism between the torsion subgroups of $F^1(X)$ and $\Alb_X(\kk)$. This in particular implies that $F^2(X)$ is torsion-free. This information will be used heavily in the proofs of the main results. Namely, in order to show that a certain zero-cycle $z\in F^2(X)$ vanishes, it is enough to show $z$ is torsion. 
  
\subsection{Hyperelliptic points}\label{sec:hyperelliptic} Let $A$ be an abelian surface over $\kk$. We introduce the notion of a \textit{hyperelliptic point}, which will be used throughout the article. For the purposes of this paper, a \emph{hyperelliptic curve} is a smooth projective curve $H$ over $\kk$ of genus $g\geq 1$\footnote{We allow hyperelliptic curves to be genus $1$ in order to simplify the proofs of results such as \autoref{Tschinkelreprove} and \autoref{easypoints}. In these cases it is important to remember that the hyperelliptic involution and the set of Weierstrass points depend on the choice of map $H\to\mathbb{P}^1$.} equipped with a prescribed degree $2$ map $H\to\mathbb{P}^1$. Any such curve comes equipped with a \emph{hyperelliptic involution} $\iota \colon H\to H$ determined by the property that $H\to\mathbb{P}^1$ is invariant under $\iota$, and the fixed points of $\iota$ in $H(\kk)$ are \emph{Weierstrass points} of $H$. 

\begin{defn}\label{goodpoint} A point $a\in A(\kk)$ is called hyperelliptic if there exists a hyperelliptic curve~$H$ and a morphism $f \colon H\to A$ with $f\circ \iota=-f$ such that a nonzero multiple of $a$ lies in the image of $f$. 
\end{defn}
 
\begin{lem}\label{Weierstrass} Let $f \colon H\to A$ be a morphism from a hyperelliptic curve $H/\kk$ to $A$. Then $f\circ\iota=-f$ \black if and only if there exists a  Weierstrass point of $H$ mapping under $f$ to a $2$-torsion point of $A$. 
\end{lem}

\begin{proof}
	Fix a base point $q\in H(\kk)$, so that an embedding $H\to \Jac(H)$ of $H$ in its Jacobian variety is determined by $p\mapsto [p]-[q]$. By the universal property of $\Jac(H)$, $f$ factors through this embedding. Since every morphism of abelian varieties is a homomorphism followed by translation, there exists a homomorphism $\psi \colon \Jac(H)\to A$ such that
	\[f(p)=\psi([p]-[q])+f(q).\]
	Now for any $p\in H(\kk)$, the divisor $[\iota p]+[p]-[\iota q]-[q]$ is principal, so
	\begin{align*}
		f(\iota p)+f(p)&=\psi([\iota p]+[p]-2[q])+2f(q)\\
		&=\psi([\iota q]-[q])+2f(q)\\
		&=f(\iota q)+f(q).
	\end{align*}
	Thus the value of $f(\iota p)+f(p)$ is constant\footnote{This also follows immediately from Weil reciprocity (\autoref{eq:WR} with $r=1$).} for all $p\in H(\kk)$, and we have $f\circ\iota=-f$ if and only if this constant is zero. In particular, 
	 for a Weierstrass point $q$, we have \black $f\circ\iota=-f$ if and only if $f(q)$ is $2$-torsion.
	 
\end{proof}
 The proof of the above lemma in fact shows that if $f\circ\iota=-f$, then $f$ maps every Weierstrass point of $H$ to a $2$-torsion point of $A$.  

\autoref{goodpoint} and \autoref{Weierstrass} clearly generalize to higher dimensional abelian varieties and will be used in \autoref{abvars}. The following lemma which is 
 unique to \black abelian surfaces 
can also be found \black in \cite[Lemma 4.1]{bogomolov_tschinkel}.  

For an abelian surface $A$ over $\kk$ we consider the quotient $A/\langle -1\rangle$ by the negation action. This surface has sixteen singularities corresponding to the sixteen $2$-torsion points of $A$. The Kummer surface $K$ associated to $A$ is the $K3$ surface obtained by resolving the singularities of $A/\langle -1\rangle$.  
\black 
\begin{lem}\label{Tschinkelreprove} Let $K$ be the Kummer surface associated to the abelian surface $A$. Then the set of images of non-constant morphisms $\mathbb{P}^1\to K$ is in one-to-one correspondence with the set of images of morphisms $f \colon H\to A$ from hyperelliptic curves $H$ such that $f\circ\iota=-f$.
	
\end{lem}

\begin{proof}
	If $f \colon H\to A$ is constant and $f\circ \iota=-f$, then the image of $f$ is a $2$-torsion point. Each such point maps to a singularity of $A/\langle -1\rangle$ and blows up to a rational curve in $K$, so we associate to each $2$-torsion point of $A$ the corresponding blowup in $K$.
	
	Now we establish a bijection between images of non-constant maps $f \colon H\to A$ with $f\circ\iota=-f$ and images of non-constant maps $\mathbb{P}^1\to K$ whose image does not map to a singular point of $A/\langle -1\rangle$. First suppose $f \colon H\to A$ is non-constant and $f\circ\iota=-f$. The composition $H\to A\to A/\langle -1\rangle$ factors through the double cover  $H\to\mathbb{P}^1$, resulting in a non-constant map $\mathbb{P}^1\to A/\langle -1\rangle$. Since the image of this map is not contained in the singular locus, we obtain a rational map $\mathbb{P}^1\dashrightarrow K$, which extends to a morphism $\mathbb{P}^1\to K$ since $\mathbb{P}^1$ is a smooth projective curve and $K$ is projective.
	
	In the other direction, if $\mathbb{P}^1\to K$ is a non-constant map with image not contained in the exceptional locus, then $\mathbb{P}^1\to A/\langle -1\rangle$ is non-constant. Taking fiber products we obtain a diagram
	\[\begin{tikzcd}
		C\arrow[rr]\arrow[d] & & A\arrow[d,"\pi"] \\
		\mathbb{P}^1\arrow[r]& K \arrow[r] & A/\langle -1\rangle 
	\end{tikzcd}\]
	for some degree $2$ cover $C\to \mathbb{P}^1$. Let $H$ denote the normalization of $C$. Since $A$ contains no rational curves, $H$ must be irreducible and have genus at least $1$, and the composition $H\to C\to\mathbb{P}^1$ gives $H$ the structure of a hyperelliptic curve. Let $f \colon H\to C\to A$. By the diagram above we observe that $\pi\circ f\circ \iota=\pi\circ f$, and therefore $f\circ\iota=\pm f$. But $f\circ\iota=f$ would imply that $f$ factors through the map $H\to\mathbb{P}^1$, a contradiction because $A$ contains no rational curves. Hence we have $f\circ \iota=-f$ as desired.\black
	
\end{proof}

\subsection{Zero-cycles on Abelian Surfaces} Let $A$ be an abelian surface over $\kk$ with zero element $0\in A$. A straightforward computation gives the following explicit generators for the subgroups $F^1(A)$ and $F^2(A)$.
 \begin{eqnarray*}
&&F^{1}(A)=\langle[a]-[0]:a\in A(\kk)\rangle,\\
&&F^{2}(A)=\langle[a+b]-[a]-[b]+[0]:a,\;b\in A(\kk)\rangle. 
\end{eqnarray*}
\begin{notn} For $a,b\in A(\kk)$ we will denote by $z_{a,b}:=[a+b]-[a]-[b]+[0]$. 
\end{notn}
 

Given a vector $\mathbf{m}=(m_1,\ldots,m_n)\in \Z^n$ and a tuple $\mathbf{a}:=(a_1,\ldots,a_n)\in A(\kk)^n$, we define 
\[\mathbf{m}\cdot \mathbf{a}:= m_1a_1+\cdots+m_na_n\in A(\kk).\] The \emph{divisible hull} of a subgroup $\Gamma\leq A(\kk)$ is defined to be the set of elements $x\in A(\kk)$ for which some nonzero multiple of $x$ lies in $\Gamma$. In particular, if two subgroups $\Gamma_1,\Gamma_2$ are commensurable (that is, $\Gamma_1\cap \Gamma_2$ has finite index in both $\Gamma_1$ and $\Gamma_2$), then they have the same divisible hull.

The goal of this subsection is to prove the following more general version of \autoref{abel}. 
 
 \begin{theo}\label{abelsurface1} Let $A$ be an abelian surface over an algebraically closed field $\kk$ and let $\mathbf{a}=(a_1,\ldots,a_n)\in A(\kk)^n$.
 Suppose there exists a finite set $M\subseteq \Z^n$ such that
 \begin{itemize}
 	\item the set of tensors $\mathbf{m}\otimes \mathbf{m}$ for $\mathbf{m}\in M$ spans $\Sym^2\Q^n$ as a $\Q$-vector space, and
 	\item for each $\mathbf{m}\in M$, the point $\mathbf{m}\cdot\mathbf{a}$ is hyperelliptic.
 \end{itemize}
    Then for every $c,d\in B_{\mathbf{a}}$, where $B_{\mathbf{a}}$ is the divisible hull of the subgroup of $A(\kk)$ generated by $a_1,\ldots,a_n$, \black the zero-cycle 
 	\[z_{c,d}:=[c+d]-[c]-[d]+[0]\] vanishes in $F^2(A)$.
 \end{theo}
 
 Put simply, in order to prove $z_{c,d}=0$ for all linear combinations $c,d$ of $a_1,\ldots,a_n$, it suffices to show that a certain finite list of linear combinations of $a_1,\ldots,a_n$ are hyperelliptic points. \black
 
 Since the statement of the theorem is somewhat technical, we begin with a few special cases. First note that \autoref{abel} follows immediately by taking $n=2$ and $M=\{(1,0),(0,1),(m,m')\}$ for any nonzero $m,m'\in\Z$, because the tensors $(1,0)\otimes(1,0)$, $(0,1)\otimes(0,1)$, and $(m,m')\otimes(m,m')$ span $\Sym^2\Q^2$. 
 More generally, we have the following.
 \begin{cor}\label{cor:hyperelliptic pair sums}
 	 Let $A$ be an abelian surface over an algebraically closed field $\kk$ and let  $a_1,\ldots,a_n\in A(\kk)$ be hyperelliptic points. Suppose additionally that for any pair $1\leq i<j\leq n$, there exist nonzero integers $r_{i,j},s_{i,j}$ such that $r_{i,j}a_i+s_{i,j}a_j$ is hyperelliptic. Then for all $c,d$ in the divisible hull of the subgroup generated by $a_1,\ldots,a_n$, the zero-cycle $z_{c,d}$ vanishes in $F^2(A)$.
 \end{cor}
 \begin{proof}
 	Letting $e_i\in\Z^n$ denote the vector which is $1$ in the $i$-th component and $0$ otherwise, we can take 
 	\[M=\{e_i:1\leq i\leq n\}\cup\{r_{i,j}e_i+s_{i,j}e_j:1\leq i<j\leq n\}.\]
 \end{proof} 
 
 We now proceed with a proof of \autoref{abelsurface1}. \black The theorem will follow from the key lemmas \autoref{bilinear} and \autoref{odd}. 
 The first lemma is probably well-known to the experts, and it may have a simpler proof. For completeness we include a proof using some results from \cite{Gazaki2015}. 
 \begin{lem}\label{bilinear}  The map $A(\kk)\times A(\kk)\to F^2(A)$ given by $(a,b)\mapsto z_{a,b}$ is bilinear and symmetric. \black 
 \end{lem}
 \begin{proof}
 	The symmetry is clear. In \cite{Gazaki2015} the first author defined for an abelian variety $A$ over an arbitrary field $k$ a filtration $\{F^{r}(A)\}_{r\geq 0}$ extending the filtration $\CH_0(A)\supset F^1(A)\supset F^2(A)$. Among other properties, it follows by \cite[Proposition 3.4]{Gazaki2015} that there exists a well-defined homomorphism $\Psi \colon  A(k)\otimes A(k)\to F^2(A)/F^{3}(A)$, with $\Psi(a\otimes b)=z_{a,b}$. In the case of an abelian surface, it follows by \cite[Corollary 4.5]{Gazaki2015} that the group $F^3(A)$ is torsion, and hence over an algebraically closed field $\kk$ it vanishes. Thus, we have a linear map 
 	\[A(\kk)\otimes A(\kk)\to F^2(A), a\otimes b\mapsto z_{a,b},\] yielding the desired bilinearity. 
 \end{proof}
 An important consequence of \autoref{bilinear} is that there is a map $A(\kk)\otimes A(\kk)\to F^2(A)$ induced by $a\otimes b\mapsto z_{a,b}$. \black
 Since the group $F^2(A)$ is torsion-free, it follows by bilinearity that $z_{p,q}=0$ if $p$ or $q$ is a torsion point. 
 \begin{cor}\label{zaa} The Albanese kernel $F^2(A)$ vanishes if and only if $z_{a,a}=0$ for all $a\in A$.  
 \end{cor}
 \begin{proof}
 	The group $F^2(A)$ is generated by zero-cycles of the form $z_{a,b}$ with $a,b\in A$. Bilinearity and symmetry give
 	\begin{align}\label{eq:zab to hyperelliptic}
 		2z_{a,b}=z_{a,a}+z_{b,b}-z_{a+b,a+b},
 	\end{align}
 	from which the claim follows. 
 \end{proof}
  
 \begin{lem}\label{odd} Let $a\in A$ be a hyperelliptic point in the sense of \autoref{goodpoint}. Then $z_{a,a}=0$. 
 \end{lem}
 \begin{proof} By assumption there exists a morphism $f \colon H\to A$  from a hyperelliptic curve $H$ such that $f\circ \iota=-f$,  and $na$ lies in the image of $f$ for some nonzero integer $n$. Since $z_{na,na}=n^2z_{a,a}$, we can reduce to the case $n=1$.  
 	By \autoref{Weierstrass} there is a Weierstrass point $w$ of $H$ such that $f(w)$ is a $2$-torsion point of $A$. \black The proper morphism $f$ induces a pushforward homomorphism 
 	\begin{align*}
 		f_\star \colon \Pic^0(H)&\to F^1(A)\\
 		[q]-[w]&\mapsto [f(q)]-[f(w)].
 	\end{align*} \black
 	Note that $[q]+[\iota(q)]-2[w]$  is the pullback of a divisor on $\mathbb{P}^1$ and is therefore \black a principal divisor on $H$.  Letting $a=f(q)$ and $p_0=f(w)$,  the relation $f\circ\iota=-f$ \black yields
 	
 	\begin{align*}
 		0=f_\star([q]+[\iota(q)]-2[w])=[a]+[-a]-2[p_0]=-z_{a,-a}+z_{p_0,p_0}.
 	\end{align*} 
 	\black
 	But $z_{p_0,p_0}=0$ since $p_0$ is a torsion point, so $z_{a,a}=-z_{a,-a}=0$ as desired. 
 	
 \end{proof}

 \begin{proof}[Proof of \autoref{abelsurface1}] 
 	
 	Let $c,d\in B_{\mathbf{a}}$. By definition there exist integers $r,s\geq 1$ such that $rc, sd$ are $\Z$-linear combinations of $a_1,\ldots,a_n$. \black Bilinearity gives $z_{rc,sd}=rsz_{c,d}$. Since $rs\neq 0$, it is enough to show $z_{rc,sd}=0$, which allows to reduce to the case when $c,d$ lie in the subgroup $\Gamma$ generated by $a_1,\ldots,a_n$.

 	We have a homomorphism $\Z^n\to A(\kk)$ given by $\mathbf{m}\mapsto \mathbf{m}\cdot\mathbf{a}$, and the image of this map is $\Gamma$. 
 	By \autoref{bilinear}, there is a well-defined symmetric linear map $A(\kk)\otimes A(\kk)\to F^2(A)$ induced by $a\otimes b\mapsto z_{a,b}$. Now consider the composition
 	\[\Phi:\Z^n\otimes \Z^n\to A(\kk)\otimes A(\kk)\to F^2(A).\]
 	Let $\mathbf{c}\in\Z^n$ map to $c\in \Gamma$ and $\mathbf{d}\in\Z^n$ map to $d\in \Gamma$. By assumption, the symmetric tensor $\mathbf{c}\otimes \mathbf{d}+\mathbf{d}\otimes\mathbf{c}\in \Z^n\otimes \Z^n$ can be written as a rational combination of the tensors $\mathbf{m}\otimes\mathbf{m}$ for $\mathbf{m}\in M$. Clearing denominators, we have integers $t_{\mathbf{m}}$ for each $\mathbf{m}\in M$ and $u\neq 0$ satisfying
 	\[u(\mathbf{c}\otimes \mathbf{d}+\mathbf{d}\otimes\mathbf{c})=\sum_{\mathbf{m}\in M} t_\mathbf{m}(\mathbf{m}\otimes\mathbf{m}).\]
 	Applying $\Phi$ to both sides we obtain
 	\[2uz_{c,d}=\sum_{\mathbf{m}\in M} t_\mathbf{m}z_{\mathbf{m}\cdot\mathbf{a},\mathbf{m}\cdot\mathbf{a}}.\]
 	Since $\mathbf{m}\cdot\mathbf{a}$ is hyperelliptic for all $\mathbf{m}\in M$, the right-hand side vanishes by \autoref{odd}. This implies $z_{c,d}$ is torsion, and hence it vanishes as desired. 
 	\black
 \end{proof}

\subsection*{Products of Elliptic Curves}
In the case of a product $A=E\times E'$ of two elliptic curves, we can say something more.  In this case we have an equality 
   \[F^2(E\times E')=\langle [(p,q)]-[(p,0)]-[(0,q)]+[(0,0)]:p\in E_1(\kk), q\in E_2(\kk)\rangle.\] 
  \begin{lem}\label{easypoints} The points $(p,0)$ and $(0,q)$ are hyperelliptic in the sense of \autoref{goodpoint}. 
  \end{lem}

\begin{proof}
	We can give $E$ the structure of a hyperelliptic curve by letting the hyperelliptic involution $\iota \colon E\to E$ be given by negation on $E$. Then the map $f \colon E\to E\times E'$ determined by $p\mapsto (p,0)$ evidently satisfies $f\circ\iota=-f$.
\end{proof}

This allows us to state a revised version of \autoref{abelsurface1}, which can be used to prove that certain finitely generated subgroups of $F^2(A)$ are trivial. This is a direct generalization of \cite[Corollary 4.11]{GazakiLove2022}.

\begin{cor}\label{elliptic1}
	Let $A=E_1\times E_2$ be a product of elliptic curves over $\kk$. Let $\mathbf{p}=(p_1,\ldots,p_{n_1})\in E_1(\kk)^{n_1}$ and $\mathbf{q}=(q_1,\ldots,q_{n_2})\in E_2(\kk)^{n_2}$. Suppose that for each $j=1,\ldots,n_1n_2$, there exists $\mathbf{r}_j\in\Z^{n_1}$ and $\mathbf{s}_j\in \Z^{n_2}$, satisfying
	\begin{itemize}
		\item the $n_1n_2$ tensors $\mathbf{r}_j\otimes \mathbf{s}_j$ span $\Q^{n_1}\otimes \Q^{n_2}$, and 
		\item the $n_1n_2$ elements $(\mathbf{r}_j\cdot \mathbf{p},\mathbf{s}_j\cdot \mathbf{q})\in A(\kk)$ 
		are hyperelliptic points.
	\end{itemize}
	Then for every $c,d$ in the divisible hull of the group generated by $(p_1,0),\ldots,(p_{n_1},0)$, \linebreak $(0,q_1),\ldots,(0,q_{n_2})$, $z_{c,d}$ vanishes in $F^2(A)$. 
\end{cor}
\begin{proof}
	Let $n=n_1+n_2$, and identify $\Z^n$ with $\Z^{n_1}\oplus\Z^{n_2}$. Let \[\mathbf{a}=((p_1,0),\ldots,(p_{n_1},0),(0,q_1),\ldots,(0,q_{n_2}))\in A(\kk)^n.\]
	For $i=1,2$, let $M_i$ be a finite set of vectors in $\mathbf{t}\in\Z^{n_i}$ such that the set $\mathbf{t}\otimes\mathbf{t}$ for $\mathbf{t}\in M_i$ spans $\Sym^2\Q^{n_i}$ (for instance as in the proof of \autoref{cor:hyperelliptic pair sums}). Let \[M=\{(\mathbf{r}_j,\mathbf{s}_j):1\leq j\leq n_1n_2\}\cup \{(\mathbf{t},0):\mathbf{t}\in M_1\}\cup\{(0,\mathbf{t}):\mathbf{t}\in M_2\}.\]
	We will show that $M$ satisfies the conditions of \autoref{abelsurface1}. 
	
	We first show that each $\mathbf{m}\in M$ maps to a hyperelliptic point under $\mathbf{m}\mapsto \mathbf{m}\cdot\mathbf{a}$. The vector $(\mathbf{r}_j,\mathbf{s}_j)$ maps to $(\mathbf{r}_j\cdot\mathbf{p},\mathbf{s}_j\cdot\mathbf{q})$, which is hyperelliptic by assumption. For each $\mathbf{t}\in M_1$, $(\mathbf{t},0)$ maps to $(\mathbf{t}\cdot\mathbf{p},0)$, which is hyperelliptic by \autoref{easypoints}. Likewise $(0,\mathbf{t})$ maps to a hyperelliptic point for each $\mathbf{t}\in M_2$.
	
	Now we show that the vectors $\mathbf{m}^{\otimes 2}:=\mathbf{m}\otimes\mathbf{m}$ span $\Sym^2\Q^n$ as we vary over $\mathbf{m}\in M$. We have a decomposition
	\[\Sym^2\Q^n=\Sym^2(\Q^{n_1}\oplus\{0\})\oplus \Sym^2(\{0\}\oplus \Q^{n_2})\oplus V,\]
	where $V$ is the image of the linear map $\Q^{n_1}\otimes \Q^{n_2}\to \Sym^2\Q^n$ defined by $u\otimes v\mapsto (u,0)\otimes (0,v)+(0,v)\otimes (u,0)$. By construction of $M_1$ and $M_2$, the subspace $\Sym^2(\Q^{n_1}\oplus\{0\})$ is spanned by $(\mathbf{t},0)^{\otimes 2}$ for $\mathbf{t}\in M_1$, and $\Sym^2(\{0\}\oplus\Q^{n_2})$ is spanned by $(0,\mathbf{t})^{\otimes 2}$ for $\mathbf{t}\in M_2$. Now for any $u\in \Q^{n_1}$ and $v\in \Q^{n_2}$, $u\otimes v$ can be written as $\sum_{j=1}^{n_1n_2}c_j \mathbf{r}_j\otimes\mathbf{s}_j$ for some coefficients $c_j\in\Q$, and so
	\[\sum_{j=1}^{n_1n_2} c_j\big((\mathbf{r}_j,\mathbf{s}_j)^{\otimes 2}-(\mathbf{r}_j,0)^{\otimes 2}-(0,\mathbf{s}_j)^{\otimes 2}\big)=(u,0)\otimes (0,v)+(0,v)\otimes (u,0).\]
	Since $(\mathbf{r}_j,\mathbf{s}_j)\in M$, and $(\mathbf{r}_j,0)^{\otimes 2}$ and $(0,\mathbf{s}_j)^{\otimes 2}$ are in subspaces that have already been proven to be in the span of $\mathbf{m}^{\otimes 2}$ for $\mathbf{m}\in M$, we can conclude that $V$ is in the span of $\mathbf{m}^{\otimes 2}$ for $\mathbf{m}\in M$.
\end{proof}

As a special case, when $n_1=n_2=1$, we have that if $(np,mq)$ is hyperelliptic for some nonzero integers $n,m$, then the zero-cycle $[(p,q)]-[(p,0)]-[(0,q)]+[(0,0)]$ 
is zero in $F^2(A)$.

\black

\subsection{Some Progress towards Beilinson's Conjecture} 
In this section we focus on Beilinson's conjecture, which we recall below.
\begin{conj}\label{beilconj2} (Beilinson, \cite{Beilinson1984}) Let $X$ be a smooth projective variety over $\QQ$. Then the Albanese map is injective. That is, $F^2(X)=0$. 
\end{conj}

We illustrate how \autoref{abelsurface1} can be used to investigate \autoref{beilconj2}. Let us consider an abelian surface $A$ defined over a number field $k$, with $A_{\kk}$ its base change to the algebraic closure. By \autoref{bilinear} we have a surjective linear map $A(\kk)\otimes A(\kk)\to F^2(A_{\kk})$ induced by $a\otimes b\mapsto z_{a,b}$. We can write $A(\kk)$ as the union of $A(L)$ over all finite extensions $L/k$, and each $A(L)$ is finitely generated by the Mordell-Weil theorem.

\begin{cor}\label{beilreduce1} 
	Let $A$ be an abelian surface over a number field $k$. Suppose that for every finite extension $L/k$, if we let $r$ be the Mordell-Weil rank of $A(L)$ and $\mathbf{a}\in A(L)^{r}$ a tuple of independent elements in $A(L)$, there exists $M\subseteq \Z^r$ satisfying the assumptions of \autoref{abelsurface1}. Then Beilinson's \autoref{beilconj2} is true for $A_{\kk}$. 
\end{cor}

\begin{proof}  
	For any $c,d\in A(\kk)$, there exists a number field $L/k$ with $c,d\in A(L)$. Let $r$ be the rank of $A(L)$ and $\mathbf{a}\in A(L)^{r}$ a tuple of independent elements in $A(L)$. Then the components of $\mathbf{a}$ generate a finite-index subgroup of $A(L)$, so $A(L)$ is contained in the divisible hull $B_{\mathbf{a}}$. We can conclude by \autoref{abelsurface1} that $z_{c,d}=0$. Since $F^2(A_{\kk})$ is generated by cycles of this form, the result follows.
\end{proof}

In practice we can use the above methods to obtain computational evidence towards Beilinson's conjecture. For instance, if $\rk(A(L))=r$, then in order to prove $z_{a,b}=0$ for all $a,b\in A(L)$, it is enough to find $\Z$-linearly independent points $a_1,\ldots, a_r\in A(L)$ such that each $a_i$ is hyperelliptic, and for every $i\neq j$ there is a hyperelliptic $\Z$-linear combination $n_ia_i+m_ja_j$, for some $n_i, m_j\neq 0$ (see \autoref{cor:hyperelliptic pair sums}).

In the case when $A$ is isogenous to a product $E_1\times E_2$ of elliptic curves, using \autoref{elliptic1} we actually need less.
\begin{cor}\label{elliptic2}  
	Let $A/\QQ$ be an abelian surface isogenous to a product  $E_1\times E_2$ of two elliptic curves. Suppose that $A$ and the isogeny are defined over a number field $L$ and $\rk(E_i(L))=n_i$ for $i=1,2$.  Then there exist a set of $n_1n_2$ points in $A(L)$ such that if all $n_1n_2$ of these points are hyperelliptic points, then $z_{a,b}$ vanishes for all $a,b\in A(L)$.  
\end{cor} 
\black
\begin{proof}
	 The case $A=E_1\times E_2$ is just \autoref{elliptic1}. Namely, let $\mathbf{p}\in E_1(L)^{n_1}$ be a tuple of independent elements of $E_1(L)$, and $\mathbf{q}\in E_2(L)^{n_1}$ a tuple of independent elements of $E_2(L)$. For $j=1,\ldots,n_1n_2$, take any $\mathbf{r}_{j}\in\Z^{n_1}$ and $\mathbf{s}_{j}\in\Z^{n_1}$ such that the $n_1n_2$ tensors $\mathbf{r}_{j}\otimes \mathbf{s}_{j}$ are independent in $\Z^{n_1}\otimes \Z^{n_2}$. Then we can take the points $(\mathbf{r}_{j}\cdot\mathbf{p}, \mathbf{s}_{j}\cdot\mathbf{q})\in A(L)$. If these points are hyperelliptic then we obtain the desired result by \autoref{elliptic1}. \black


	More generally, suppose the above holds for a product $E_1\times E_2$ defined over $L$ and let $\phi:E_1\times E_2\to A$ be an isogeny of degree $n$ defined over $L$. Then there is an isogeny $\psi:A\to E_1\times E_2$ such that $\psi\circ\phi=[n]$. This implies that the subgroup $\phi((E_1\times E_2)(L))$ of $A(L)$ is of finite index; namely $na\in \phi(E_1\times E_2(L))$, for all $a\in A(L)$. Bilinearity then gives $n^2z_{a,b}=0$ for all $a,b\in A(L)$. 
\end{proof}

 In \autoref{computations} we will use the above Corollary to produce many new examples of $A/\Q$ for which we can annihilate $z_{a,b}$ for all $a,b\in A(\Q)$.


\begin{rem} 
	Let $A$ be an abelian surface over the algebraic closure $\overline{\F}_p$ of a finite field and~$K$ its associated Kummer surface. Bogomolov and Tschinkel (\cite[Theorem 1.1]{bogomolov_tschinkel}) showed that every $\overline{\F}_p$-point $x$ in the Kummer surface lies on some rational curve by  demonstrating that every point of $A$ is hyperelliptic; more precisely, they show that there is a single curve~$H$ such that every point $a\in A(\overline{\F}_p)$ is in the image of a morphism $f \colon H\to A$ with $f\circ\iota=-f$ \cite[Corollary 2.5]{bogomolov_tschinkel}.
	We do not expect an analogue of \cite[Corollary 2.5]{bogomolov_tschinkel} to hold over number fields: in the finite field setting, the Frobenius endomorphism of $A$ is used in an essential way to generate sufficiently many morphisms from $H$ to entirely fill out $A(\kk)$. However, this example illustrates the fact that hyperelliptic points --- even those coming from a single hyperelliptic curve --- can be very plentiful. If we have a much larger collection of hyperelliptic curves mapping into $A/\QQ$, as in \autoref{infinite_H:intro}, it seems plausible that there may \black be enough hyperelliptic points to prove Beilinson's conjecture.  
\end{rem}

\black 
\vspace{5pt}

\section{Producing curves in abelian surfaces}\label{curvessection} 

We continue to assume all varieties are defined over an algebraically closed field $\kk$. From now on we assume additionally that $\kk$ admits an embedding $\kk\hookrightarrow\C$.  The goal of this section is to show that when an abelian surface $A$ is isogenous to a product of elliptic curves, there is an abundance of hyperelliptic curves $f \colon H\to A$ with negation on $A$ acting as the hyperelliptic involution on $H$.

Given a morphism $f \colon V\to W$ of schemes, let $f(V)$ denote the scheme-theoretic image of~$f$; we say $f$ is \emph{birational onto its image}, if there exists a dense open $U\subseteq V$ such that $f(U)$ is open in $f(V)$ and $f|_U \colon U\to f(U)$ is an isomorphism. 
The following definition will be used so that we can avoid considering multiple curves $H$ that have the same image in $A$. \black
\begin{defn}
	Let $H/\kk$ be a smooth hyperelliptic curve with hyperelliptic involution $\iota$, and let $A$ be an abelian variety. A map $f \colon H\to A$ is \emph{compatible} if  $f$ is birational onto its image and $f\circ\iota = -f$. \black
\end{defn}

If $C\subseteq A$ is the image of any non-constant map $f \colon H\to A$ from a hyperelliptic curve $H$ with $f\circ\iota=-f$, then the normalization map $\widehat{C}\to C$ is a compatible map, and $f$ factors through the normalization (cf.~\autoref{Tschinkelreprove}).  In particular, every hyperelliptic point of $A$ (\autoref{goodpoint}) is in the image of a compatible map.
The main goal of this section is to prove the following more general version of \autoref{infinite_H:intro}. 
\begin{theo}\label{infinite_curves}
	Let $A/\kk$ be an abelian variety, and suppose there is a homomorphism with finite kernel from a product of two elliptic curves into $A$. For infinitely many positive integers $g\geq 2$, there exist infinitely many non-isomorphic genus $g$ hyperelliptic curves with maps $f \colon H\to A$ such that  for any abelian variety $B/\kk$ and isogeny $\mu \colon A\to B$, the composition $\mu\circ f$ is compatible. \black
\end{theo}

The proof is given in \autoref{infcurves_proof} after establishing a series of lemmas. The most important step is the following construction.

\begin{prop}\label{lem:construction}
	Let $E,E'$ be elliptic curves over $\kk$, and $n$ a positive integer. For some positive integer $g\geq 2$ satisfying $\frac16n-10\leq g\leq 4n-2$, there exists a genus $g$ hyperelliptic curve with a compatible map $f \colon H\to E\times E'$, such that the degrees of the projection maps $H\to E$ and $H\to E'$ are equal to $2n$.
\end{prop}

It is likely that $\frac16n-10$ can be replaced with a much larger lower \black bound on $g$; see \autoref{unexpectedintersection}. 

 \black

The proof of this result is the focus of \autoref{Kummerfibration} and \autoref{genusH}, but we summarize the argument here. 
Consider the quotient of $E\times E'$ by negation. Under this map, the $2$-torsion of $E\times E'$ map to singularities; the Kummer surface $K$ of $E\times E'$ is obtained by blowing up these sixteen points.  We give $K$ the structure of an elliptic fibration $\xi \colon K\to\mathbb{P}^1$, and produce a large collection of sections of this fibration using addition in the Mordell-Weil lattice. \black
By pulling back such a section $\phi \colon \mathbb{P}^1\to K$ along the quotient of $E\times E'$ by negation, we obtain a curve $C\to E\times E'$, and the normalization $H$ of $C$ is a hyperelliptic curve mapping compatibly to $E\times E'$. To summarize, we obtain the following diagram, where the vertical maps are degree $2$ and the horizontal maps are birational onto their images:
\begin{equation}\label{eq:Htikz}
	\begin{tikzcd}
		H\ar{r} & C\ar{rr}\ar{d} & & E\times E'\arrow{d} \\
		& \mathbb{P}^1\ar["\phi"]{r} & K\ar{r} & (E\times E')/\langle -1\rangle.
	\end{tikzcd}
\end{equation} Let $D$ be the divisor of $K$ consisting of the sixteen lines mapping to the singularities of $(E\times E')/\langle -1\rangle$. \black
Away from the locus of points mapping to $D$, every point of $\mathbb{P}^1$ pulls back to two smooth points of $C$. So the number of Weierstrass points of~$H$ is bounded above by the intersection multiplicity $\im\phi\cdot D$, which we compute in \autoref{lem:intersection_formula}. \black On the other hand, if $\im\phi$ intersects $D$ transversely at a point $P\in K(\kk)$, then the preimage of $P$ in $C$ is smooth, and so corresponds to a Weierstrass point of $H$. Thus the number of Weierstrass points of~$H$ is bounded below by the number of transverse intersections between $\im\phi$ and~$D$, which we compute a lower bound for in \autoref{bound_transverse}. \black We can therefore compute upper and lower bounds on the genus $g$ of $H$. Finally, the degree of the projection $H\to E$ can be computed as the intersection multiplicity in $E\times E'$ of the image of $H$ with a fiber over a point of $E$.

\begin{rem}\label{curvesoverk} 
Let $E, E'$ be elliptic curves defined over a number field $L$. If we assume both curves are in Legendre form (\autoref{eq:legendre}) over $L$, then the construction carried out in \autoref{lem:construction} produces  hyperelliptic curves $H$ over $L$ of arbitrarily large genus  with compatible maps $H\to A$ defined over~$L$.  Even under the weaker assumption $A[2]\subseteq A(L)$, a modified version of the construction still allows us to produce sections of the fibration and their hyperelliptic pullbacks. In \autoref{computations} we will use this modified version for computations over $L=\Q$.  
\end{rem}

\subsection{Kummer surface as elliptic fibration}\label{Kummerfibration}

In this section we provide a brief overview of how to give the Kummer surface of a product of elliptic curves the structure of an elliptic fibration, and then we list some facts about sections of this fibration. For more details on this construction, see \autoref{appendix:intersection}. We follow the setup described in Sections 3 to 5 of \cite{shioda07}; another summary can be found in \cite{kuwatashioda}.

Since $\kk$ has characteristic $0$, every elliptic curve over $\kk$ is isomorphic to a curve in \emph{Legendre form},
\begin{align}\label{eq:legendre}
	E_\lambda \colon y^2=x(x-1)(x-\lambda)
\end{align}
for some $\lambda\in\kk-\{0,1\}$. Let $E_a,E_b$ be elliptic curves over $\kk$ for some $a,b\in\kk-\{0,1\}$, with respective identity points $o_a,o_b$. Then the points
\begin{align*}
	p_0&=o_a, & p_1&=(0,0), & p_2&=(1,0), & p_3&=(a,0),\\
	q_0&=o_b, & q_1&=(0,0), & q_2&=(1,0), & q_3&=(b,0).
\end{align*}
exhaust the $2$-torsion points $p_i\in E_a[2]$ and $q_i\in E_b[2]$. Each point $(p_i,q_j)$ maps to a singularity in $(E_a\times E_b)/\langle -1\rangle$, and the Kummer surface $K$ is obtained by blowing up these sixteen singularities to obtain rational curves $A_{ij}$.\footnote{ The notation $A_{ij}$ is following \cite{shioda07} and \cite{kuwatashioda}. We emphasize that while $A$ is used to refer to an abelian variety elsewhere in the paper, $A_{ij}$ is used to refer to a rational curve in $K$; the difference should be clear from both the double-subscript and context.} \black The divisor
\[D:=\sum_{0\leq i,j\leq 3} A_{ij}\]
is the exceptional divisor of the map $K\to (E_a\times E_b)/\langle -1\rangle$.

An affine chart for $K$ is given by the equation
\begin{equation}\label{affinepatch}
	x_1(x_1-1)(x_1-a)t^2=x_2(x_2-1)(x_2-b),
\end{equation}
with the rational map $E_a\times E_b\dashrightarrow K$ given by 
\[(x_1,y_1,x_2,y_2)\mapsto \left(x_1,x_2,\frac{y_2}{y_1}\right).\]
We have a map $\xi \colon K\to\mathbb{P}^1$ defined by $(x_1,x_2,t)\mapsto t$; this determines a fibration of $K$ called \emph{Inose's pencil}.  

\begin{rem}\label{rem:inose_invariant}
	Suppose we have isomorphisms $\psi_1 \colon E_a\to E_{a'}$ and $\psi_2 \colon E_b\to E_{b'}$. Any such isomorphism $\psi_i$ must be of the form $\psi(x_i,y_i)=(u_i^2x_i+v_i,u_i^3y_i)$ for some $u_i\in\kk^\times$ and $v_i\in\kk$, so the isomorphism $E_a\times E_b\to E_{a'}\times E_{b'}$ descends to an isomorphism on Kummer surfaces mapping
	\[(x_1,x_2,t)\mapsto \left(u_1^2x_1+v_1,\,u_2^2x_2+u_2,\,\frac{u_2^3}{u_1^3}t\right).\]
	The fiber over $t$ on the left is sent to the fiber over $\frac{u_2^3}{u_1^3}t$ on the right, showing that Inose's pencil does not depend on the choices of Legendre forms. This can also be seen by giving a more intrinsic, coordinate-free definition of Inose's pencil, as in \cite{shioda07}. Therefore we may talk about Inose's pencil on the Kummer surface of $E\times E'$ even if the elliptic curves $E,E'/\kk$ are not given in Legendre form.
\end{rem}
\black

For $1\leq i,j\leq 3$, the curve $A_{ij}$ is a section of the fibration. If we fix $A_{11}$ as the zero section,~$K$ obtains the structure of an elliptic fibration.  The sections $A_{22},A_{33},A_{23},A_{32}$ generate a rank $4$ lattice in the corresponding Mordell-Weil group, and $A_{12},A_{13},A_{21},A_{31}$ are linear combinations of these (\autoref{appendix:intersection}). When $E_a$ and $E_b$ are non-isogenous over $\kk$, the sections $A_{22},A_{33},A_{23},A_{32}$ form a basis for the entire Mordell-Weil group (\cite[Theorem 14.6.1]{shioda07}). \black

We use $\oplus$ to denote Mordell-Weil addition, to distinguish it from the sum of curves as divisors.
 Given integers $p,q,r,s$, we define the section
\begin{align}\label{eq:Psection}
	\mathcal{P}:=pA_{22}\oplus qA_{33}\oplus rA_{23}\oplus sA_{32},
\end{align}
and set
\[n(p,q,r,s):=p(p-1)+q(q-1)+r(r-1)+s(s-1)+pq+rs.\]
Let $\cdot$ denote the intersection multiplicity of divisors on $K$.

\begin{lem}\label{lem:intersection_formula}
	Suppose $E_a,E_b$ are not isomorphic to each other, and neither has $j$-invariant~$0$. Fix $p,q,r,s\in\Z$, and let $\mathcal{P}$ be as defined in \autoref{eq:Psection}. Then
	\[\mathcal{P}\cdot D=8n(p,q,r,s)-2.\]
\end{lem}

The proof is an intersection theory computation, following \cite[Section 6]{Scholten} which handles the special case $p=q=1$, $r=s=0$. See \autoref{appendix:intersection} for details.

\begin{lem}\label{lem:projection_degree}
	Suppose $E_a,E_b$ are not isomorphic to each other, and neither has $j$-invariant~$0$. Fix $p,q,r,s\in\Z$, and let $\mathcal{P}$ be as defined in \autoref{eq:Psection}. For either $i=1$ or $2$, let $D'$ denote a fiber of the map $K\to \mathbb{P}^1$ given by $(x_1,x_2,t)\mapsto x_i$. Then
	\[\mathcal{P}\cdot D'=2n(p,q,r,s).\]
\end{lem}

The proof of this lemma is also an intersection theory computation which can be found in \autoref{appendix:intersection}.

\begin{lem}\label{cor:n_representable}
	For any positive integer $n$, there exist integers $p,q,r,s$ with $n(p,q,r,s)=n$.
\end{lem}

In fact, the number of quadruples $(p,q,r,s)\in\Z^4$ satisfying $n(p,q,r,s)=n$ is equal to $3\sigma_1(3n+2)$, where $\sigma_1(n)$ denotes the sum of positive integer divisors of $n$. This can be proven by defining a theta series whose $n$-th coefficient counts the number of quadruples with $n(p,q,r,s)=n$, and applying identities from \cite{borweinborwein} and \cite{wang2016}.

\subsection{Bounding the genus of $H$}\label{genusH}

Let $\phi \colon \mathbb{P}^1\to K$ be birational onto its image, and recall from \autoref{eq:Htikz} that we can produce a curve $C\to E_a\times E_b$ by pullback along $\pi$, the quotient of $E_a\times E_b$ by negation:
\begin{equation}
	\begin{tikzcd}
		C\ar{rr}\ar{d} & & E_a\times E_b\arrow["\pi"]{d} \\
		\mathbb{P}^1\ar["\phi"]{r} & K\ar{r} & (E_a\times E_b)/\langle -1\rangle.
	\end{tikzcd}
\end{equation}

Our goal is to compute the genus of the normalization $H\to C$, and we do this by counting Weierstrass points. Ramification points of the map $C\to\mathbb{P}^1$ do not necessarily correspond to Weierstrass points of the normalization; in particular, a singular ramification point in $C$ may have two distinct preimages in $H$ that are sent to each other by the hyperelliptic involution. However, any ramification point in the smooth locus of $C$ will lift to a Weierstrass point on~$H$.

 Recall that $D=\sum A_{ij}$ is the divisor in $K$ consisting of the blowups of the singularities of $(E_a\times E_b)/\langle -1\rangle$. 

\begin{lem}\label{smoothpreimage}
	Let $\phi$ and $C$ be as above, and $x\in \mathbb{P}^1(\kk)$. If $\im\phi$ intersects $D$ transversely at $\phi(x)$, then $x$ has a unique smooth preimage in $C$.
\end{lem}
\begin{proof}
	Since $\phi(x)\in D$, $x$ maps to a singularity of $(E_a\times E_b)/\langle -1\rangle$, which is the image of a unique $2$-torsion point of $E_a\times E_b$. So the preimage of $x$ in $C(\kk)$ is a single point $y$.
	
	Now consider a tangent vector to $\mathbb{P}^1$ at $x$. Since $\phi$ is a section of a fibration, it acts injectively on tangent spaces, so $\im\phi$ has a well-defined tangent in $K$ at $\phi(x)$. Since $\im\phi$ intersects $D$ transversely, the image in $(E_a\times E_b)/\langle -1\rangle$ also has a well-defined tangent vector. Since pullback along $\pi$ corresponds to adjoining a square root to the function field, there is a well-defined uniformizer at $y\in C(\kk)$ corresponding to the square root of the uniformizer at $x$; this proves $C$ is smooth at $y$.
	
\end{proof}

As mentioned in the outline in the beginning of \autoref{curvessection}, for a non-torsion section~$\mathcal{P}$ we want to give a lower bound on the number of transverse intersections between~$\mathcal{P}$ and~$D$. Notice that this number is bounded below by the number of transverse intersections between~$\mathcal{P}$ and the identity section $A_{11}$.  

\begin{lem}\label{bound_transverse}
	Suppose $E_a,E_b$ are not isomorphic to each other, and neither has $j$-invariant~$0$. Let $\mathcal{P}$ be defined by \autoref{eq:Psection}, with $(p,q,r,s)\neq (0,0,0,0)$. The number of transverse intersections between $\mathcal{P}$ and the identity section $A_{11}$ is at least  $\frac13n(p,q,r,s)-18$.
\end{lem}

The proof of this lemma uses results from \cite{ulmerurzua} on the \emph{Betti foliation} of $K$. Specifically, Ulmer and Urzua prove an upper bound on the sum of local intersection multiplicities between a non-torsion section of an elliptic fibration and a leaf of the Betti foliation of the fibration. In our case, $\mathcal{P}$ is a non-torsion section, and the identity section $A_{11}$ is a leaf of the foliation. If too many of the intersection points between $\mathcal{P}$ and $A_{11}$ had multiplicity greater than $1$, Ulmer and Urzua's upper bound would be broken. See \autoref{appendix:intersection} for the full computation.

We now recall the statement of \autoref{lem:construction}.

\begin{prop}\label{genusbounds}
	Let $E,E'$ be elliptic curves over $\kk$, and $n$ a positive integer. For some positive integer $g\geq 2$ satisfying $\frac16n-10\leq g\leq 4n-2$, there exists a genus $g$ hyperelliptic curve with a compatible map $f \colon H\to E\times E'$, such that the degrees of the projection maps $H\to E$ and $H\to E'$ are equal to $2n$.
\end{prop}

\begin{proof}
	 Let $K$ denote the Kummer surface of $E\times E'$; we give $K$ the structure of an elliptic fibration as in \autoref{Kummerfibration}. \black By \autoref{cor:n_representable}, there exist $p,q,r,s\in\Z$ satisfying $n(p,q,r,s)=n$; use these to define the section $\mathcal{P}$ as in \autoref{eq:Psection}. By  \autoref{lem:intersection_formula} we have $\mathcal{P}\cdot D=8n-2$, so there are at most $8n-2$ points of $\mathbb{P}^1$ that pull back to Weierstrass points on $H$. On the other hand, the number of transverse intersections between $\mathcal{P}$ and $D$ is at least the number of transverse intersections between $\mathcal{P}$ and $A_{11}$, which is greater than or equal to $\frac13n-18$ by \autoref{bound_transverse}. By \autoref{smoothpreimage}, each transverse intersection between $\mathcal{P}$ and $D$ lifts to a smooth ramification point in $C$, and therefore corresponds to a Weierstrass point of $H$. Thus the number of Weierstrass points of $H$ is at least $\frac13n-18$ and at most $8n-2$. 	
	The number of Weierstrass points also equals $2g+2$, where $g$ is the genus of $H$. Thus the genus of $H$ satisfies
	\[\frac16n-10\leq g\leq 4n-2.\]
	
	By \autoref{lem:projection_degree}, the intersection multiplicity of $\mathcal{P}$ with a fiber of the map $K\to \mathbb{P}^1$ given by $(x_1,x_2,t)\mapsto x_1$ is $2n$. \black Pulling back to $E_a\times E_b$, we find that the intersection multiplicity of $C$ with a fiber of the map to $\mathbb{P}^1$ given by taking the $x_1$ coordinate is equal to $4n$. Thus the intersection multiplicity with a fiber of the projection map to $E_1$ is equal to $2n$, and this computes the degree of the projection to $E_1$. Likewise, the degree of the projection to $E_2$ is~$2n$.
	
\end{proof}

\begin{rem} 
	The assumption $\kk\hookrightarrow\C$ is only used in \autoref{bound_transverse}: the bounds on transverse intersections \black in \cite{ulmerurzua-transversality} use analytic methods that only hold for varieties defined over the complex numbers. If a similar bound holds more generally (as we expect, cf.~\autoref{unexpectedintersection}), then \autoref{infinite_curves} holds over any algebraically closed field of characteristic $0$. If we further require in \autoref{infinite_curves} that the two elliptic curves are ordinary, that the homomorphism from the product of two elliptic curves into $A$ is \'etale onto its image, and that the isogeny $\mu \colon A\to B$ is \'etale, then the result holds over any algebraically closed field of characteristic not equal to $2$ or $3$.
\end{rem}
 
\subsubsection{Improving the lower bound}\label{unexpectedintersection}
	While the proof  of \autoref{genusbounds}  gives us a compatible map from a hyperelliptic curve with genus $g$ in the interval $[\frac16n-10,4n-2]$, it is straightforward to obtain much tighter lower bounds on $g$. For instance, \autoref{bound_transverse} can be improved by bounding tangencies with the sections $A_{ij}$ for other $1\leq i,j\leq 3$.
	In fact, we predict that the lower bound can be replaced with $4n-2$. This would imply that for every $g\equiv 2\pmod 4$, there are infinitely many hyperelliptic curves $H$ with genus $g$  satisfying the conditions of \autoref{infinite_curves}. 
	
	Our expectation is based on the observation that tangencies between $\mathcal{P}$ and $D$ should be very unusual (they are an ``unexpected intersection;'' see Remark 1.2 and Theorem 1.7 of \cite{ulmerurzua-transversality}). So heuristically, for most choices of $E$, $E'$, and $\mathcal{P}$, all intersections of $\mathcal{P}$ with $D$ should be transverse. 
	This heuristic appears to be borne out by evidence. For instance, let $\mathcal{L}$ denote the set of the first $100$ elliptic curves over $\Q$ of rank $1$ and $E[2](\Q)\simeq (\Z/2\Z)^2$, ordered by conductor (obtained from \cite{lmfdb}). \black We considered all pairs $(E,E')\in \mathcal{L}^2$ such that $E$ and $E'$ are not isomorphic over $\QQ$. For each of these pairs, we considered all sections $\mathcal{P}$ with $1\leq n\leq 3$ (there are $99$ of these), and from each of these sections constructed the associated hyperelliptic curve $H$.

	Of the $4938$ pairs of curves, there are $4571$ pairs ($\approx 93\%$) with the property that for all $99$ of the sections considered, the genus of the associated hyperelliptic curve $H$ equals the expected value $4n-2$. For the remaining $355$ pairs of curves, there are anywhere from $1$ to $29$ sections for which the genus of $H$ is lower than expected (between $4n-5$ and $4n-3$ inclusive), and the remaining sections produce $H$ with genus $4n-2$. Understanding the circumstances under which $H$ has lower genus than expected is a topic for further exploration. However, we note that in every pair of curves we checked, the majority of sections considered produce curves $H$ with genus equal to the expected $4n-2$.  \black

\subsection{Compatible maps and isogenies}\label{sec:birat_etale}

In the setup of \autoref{infinite_curves}, we have an isogeny from a product of elliptic curves $E\times E'$ to an abelian subvariety of $A$; applying \autoref{lem:construction}, we obtain a compatible map $f \colon H\to E\times E'$. In this section we show that the composition of these maps is compatible. In fact, we will show that this map $f$ remains compatible after composition with any isogeny $\mu$ from $E\times E'$.

The idea is to show that if $\mu\circ f$ is not compatible, then the curve $f(H)$ must be preserved under translation by some nonzero $p\in (E\times E')(\kk)$. We will rule out this possibility in the following lemmas. For $p\notin (E\times E')[2]$, the argument is relatively straightforward: it depends on the observation that the map $x\mapsto x+p$ does not commute with negation on $E\times E'$, but if it preserves $f(H)$, then its restriction to $f(H)$ must commute with the hyperelliptic involution on $H$. 

Given an abelian variety $A/\kk$ (not necessarily the same $A$ as in \autoref{infinite_curves}) and a point $p\in A(\kk)$, let $\tau_p$ denote the translation map $x\mapsto x+p$ on $A$.
\black

\begin{lem}\label{lem:only_2_translates}
	Let $A/\kk$ be an abelian variety, $H/\kk$ a hyperelliptic curve of genus $g\geq 2$, \black and $H\xrightarrow{f} A$ a compatible map. For all $p\in A(\kk)\setminus A[2]$, $\tau_p\circ f$ and $f$ do not have the same image.
\end{lem}

\begin{proof}
	For the sake of contradiction, let $p\in A(\kk)\setminus A[2]$ and suppose $f(H)$ is preserved under translation by $p$. This induces an automorphism $\sigma$ of $H$ satisfying $f\circ\sigma=\tau_p\circ f$. 
	Recall that under the assumption of compatibility, we have the setup
	\[\begin{tikzcd}
		H\arrow{r}{f}\arrow{d}[swap]{\pi_H} & A\arrow{d}{\pi_A} \\
		\mathbb{P}^1\arrow{r}{\phi} & A/\langle-1\rangle,
	\end{tikzcd}\]
	where $\phi$ maps $\mathbb{P}^1$ birationally onto its image. Combining this with the automorphisms $\tau_p$ and $\sigma$ above, we obtain a commutative diagram
	\[\begin{tikzcd}
		H\arrow{r}{f}\arrow{d}[swap]{(\pi_H,\pi_H\circ\sigma)} & A\arrow{d}{(\pi_A,\pi_A\circ \tau_p)} \\
		\mathbb{P}^1\times \mathbb{P}^1\arrow{r}{(\phi,\phi)} & A/\langle-1\rangle\times A/\langle-1\rangle.
	\end{tikzcd}\]
	Let $C$ denote the image of $H$ in $\mathbb{P}^1\times\mathbb{P}^1$. Since every automorphism of a hyperelliptic curve  of genus $g\geq 2$ \black commutes with the hyperelliptic involution (see for instance \cite{shaska}), $\sigma$ descends to an automorphism of $\mathbb{P}^1$, and therefore $C$ is a rational curve, parametrized by the first component.
	
	We now show that $(\pi_A,\pi_A\circ \tau_p)$ maps $A$ birationally onto its image. We first show the map is injective on $A(\kk)$. Assume 
	\[(\pi_A(x),\pi_A(x+p))=(\pi_A(y),\pi_A(y+p))\]
	for some $x,y\in A(\kk)$. From the first component we obtain either $x=y$ (in which case we are done) or $x=-y$. Consider the case $x=-y$, so that we have $x\not\in A[2]$. From the second component we obtain either $x+p=y+p$  or $x+p=-(y+p)$. Assuming $x+p=y+p=-x+p$ gives $2x=0$, which contradicts that $x\not\in A[2]$. Similarly, assuming $x+p=-(y+p)$, we find $2p=0$, contradicting our assumption $p\notin A[2]$. Hence $(\pi_A,\pi_A\circ \tau_p)$ is injective on $A(\kk)$. Since it is \'{e}tale away from $A[2]\cup \tau_p^{-1}(A[2])$, it is an isomorphism away from a finite locus of points. Thus it has a partial inverse $\psi$ that is well-defined outside of a dimension zero locus.
	
	Since $\phi$ is a birational map, its image in $A/\langle-1\rangle$ is one-dimensional, and therefore the map $\psi\circ (\phi,\phi)$ sends a nonempty open subset of the rational curve $C$ isomorphically into the abelian variety $A$. This is a contradiction as abelian varieties contain no rational curves.
	
\end{proof}

\begin{rem}
	This result is false if we allow genus $1$ hyperelliptic curves (cf.\ \autoref{sec:hyperelliptic}): as an example, for any two elliptic curves $E,E'/\kk$, the image of the compatible map $f \colon E\to E\times E'$ given by $x\mapsto (x,0)$ is preserved under translation by any element of $E(\kk)\times \{0\}$. The reason the proof does not apply in this context is that most translation maps on $E$ do not commute with negation on $E$.
\end{rem}
\black

The case $p\in A[2]\setminus\{0\}$ is trickier because $\tau_p$ does commute with negation and so descends to an involution $\overline{\tau_p}$ on $K$. 
To handle this case, we return to the special setting of compatible maps $H\to E\times E'$ that arise as pullbacks from sections $\phi \colon \mathbb{P}^1\to K$ of Inose's pencil. That is, if we pick Legendre models $E_a$ and $E_b$ for $E$ and $E'$ respectively (note that Inose's pencil is independent of the choice of Legendre models by \autoref{rem:inose_invariant}), then in terms of the affine chart 
\[x_1(x_1-1)(x_1-a)t^2=x_2(x_2-1)(x_2-b)\]
for $K$, we have $\phi(t)=(\phi_1(t),\phi_2(t),t)$ for some rational functions $\phi_1,\phi_2$.

\begin{lem}\label{lem:notranslates}
	Let $E,E'/\kk$ be elliptic curves, and let $f \colon H\to E\times E'$ be a compatible map whose image in the Kummer surface of $E\times E'$ is a section of Inose's pencil. For all $p\in (E\times E')(\kk)\setminus \{0\}$, $\tau_p\circ f$ and $f$ do not have the same image.
\end{lem}
\begin{proof}
	Let $A=E\times E'$. By \autoref{lem:only_2_translates} it suffices to consider the case $p\in A[2]\setminus\{0\}$. Since~$\tau_p$ commutes with negation on $A$, it descends to an involution $\overline{\tau_p}$ on $K$. We will show that~$\overline{\tau_p}$ does not preserve any section of Inose's pencil, which will then show that $\tau_p$ does not preserve the pullback to $A$ of any such section.
	
	By choosing the isomorphisms $E\simeq E_a$ and $E'\simeq E_b$ to Legendre models appropriately, we can ensure that each component of $p$ (now identified with its image in $E_a\times E_b$) is either the identity or $(0,0)$. This allows us to consider just two cases: $p=((0,0),o_b)$ (the case $p=(o_a,(0,0))$ is similar) and $p=((0,0),(0,0))$. We note that by the addition laws on $E_a,E_b$, for any $(x_1,y_1)\in E_a(\kk)$ and $(x_2,y_2)\in E_b(\kk)$ we have
	\[(x_1,y_1)+_{E_a}(0,0)=\left(\frac{a}{x_1},\,\frac{ay_1}{x_1^2}\right),\qquad (x_2,y_2)+_{E_b}(0,0)=\left(\frac{b}{x_2},\,\frac{by_2}{x_2^2}\right).\] 
	
	\vspace{10pt}
	
	\noindent \textbf{Case 1:} Let $p=((0,0),(0,0))$. Then $\tau_p$ descends to an involution $\overline{\tau_p}$ on $K$ given by
	\[\overline{\tau_p} \colon (x_1,x_2,t)\mapsto \left(\frac{a}{x_1},\frac{b}{x_2},t\frac{bx_1^2}{ax_2^2}\right).\]
	For contradiction, suppose $\phi \colon \mathbb{P}^1\to K$ is a section of Inose's pencil that is preserved by this involution; this induces an involution $\sigma$ on $\mathbb{P}^1$ such that $\phi\circ\sigma=\overline{\tau_p}\circ\phi$. In particular,
	\[\sigma(t)=t\frac{b\phi_1(t)^2}{a\phi_2(t)^2}\]
	where $\phi_1,\phi_2$ are the rational functions defining $\phi(t)=(\phi_1(t),\phi_2(t),t)$. Since $\sigma$ is an automorphism of $\mathbb{P}^1$, we have $\displaystyle\sigma(t)=\frac{c_{1} t+c_{2}}{c_{3} t+c_{4}}$ for some $c_{i}\in\kk$. We can conclude that $\frac{\sigma(t)}{t}=\frac{c_{1} t+c_{2}}{c_{3} t^2+c_{4} t}$ is the square of a non-constant rational function in $t$;   by degree considerations we must have
	\[\frac{b\phi_1(t)^2}{a\phi_2(t)^2}=\frac{\sigma(t)}{t}=\frac{c}{t^2}\]
	for some $c\in\kk^\times$. Thus we have  $\displaystyle\phi_1(t)=\frac{\kappa}{t}\phi_2(t)$ for some square root $\kappa$ of $\frac{ac}{b}\in\kk^\times$. 
	
	Now write $\phi_2(t)=\frac{p(t)}{q(t)}$ for polynomials $p(t),q(t)\in\kk[t]$ with no common factors. Since $(\phi_1(t),\phi_2(t),t)$ lands in $K$, we have
	\begin{align*}
		\frac{\kappa p(t)}{tq(t)}\left(\frac{\kappa p(t)}{tq(t)}-1\right)\left(\frac{\kappa p(t)}{tq(t)}-a\right)t^2&=\frac{p(t)}{q(t)}\left(\frac{p(t)}{q(t)}-1\right)\left(\frac{p(t)}{q(t)}-b\right)\\
		\Rightarrow\qquad\kappa\left(\kappa p(t)-tq(t)\right)\left(\kappa p(t)-atq(t)\right)&=t\left(p(t)-q(t)\right)\left(p(t)-bq(t)\right).
	\end{align*} 
	We must have $\kappa^3p(t)^2\equiv 0\pmod t$ and therefore $p(t)=t\tilde{p}(t)$ for some $\tilde{p}(t)\in\kk[t]$. But then 
	\[\kappa t\left(\kappa \tilde{p}(t)-q(t)\right)\left(\kappa \tilde{p}(t)-aq(t)\right)=\left(t\tilde{p}(t)-q(t)\right)\left(t\tilde{p}(t)-bq(t)\right),\] 
	which implies $0\equiv bq(t)^2\pmod t$, so $t\mid q(t)$. This contradicts the assumption that $p(t),q(t)$ have no common factors.
	
	\vspace{10pt}
	
	\noindent \textbf{Case 2:} Let $p=((0,0),o_b)$. Then $\tau_p$ descends to an involution $\overline{\tau_p}$ on $K$ given by
	\[\overline{\tau_p} \colon (x_1,x_2,t)\mapsto \left(\frac{a}{x_1},x_2,t\frac{x_1^2}{a}\right).\]
	For contradiction, suppose the image of $\phi$ is preserved by $\overline{\tau_p}$. The same argument as above shows that 
	\[\frac{\phi_1(t)^2}{a}=\frac{\sigma(t)}{t}=\frac{c}{t^2}\]
	for some $c\in\kk^\times$. Thus $\phi_1(t)=\frac{\kappa}{t}$ for some square root $\kappa$ of $ac$. Writing $\phi_2(t)=\frac{p(t)}{q(t)}$ for relatively prime $p(t),q(t)\in\kk[t]$, we have
	\[q(t)^3\kappa (\kappa - t)(\kappa -at) = tp(t)(p(t)-q(t))(p(t)-bq(t)).\]
	Since $q(t)$ is relatively prime to $p(t)$, it is also relatively prime to $p(t)-q(t)$ and to $p(t)-bq(t)$. Hence $q(t)^3\mid t$, which implies $q(t)$ is constant. But then the right-hand side is divisible by~$t$ while the left-hand side is not, again a contradiction.
	
\end{proof}
\black

\begin{prop}\label{prop:multcompatible}
 Let $f \colon H\to E\times E'$ be a compatible map whose image in the Kummer surface of $E\times E'$ is a section of Inose's pencil, and let $\mu \colon E\times E'\to B$ be an isogeny. Then the map $\mu\circ f$ is compatible. 
\end{prop}
	
\begin{proof}
	 Let $V\subseteq H\times  \ker\mu$ denote the subvariety determined by the conditions that $(x,p)\in V(\kk)$ if and only if $f(x)+p\in f(H(\kk))$.  This implies that for any $x,y\in H(\kk)$ with $\mu(f(x))=\mu(f(y))$, we have  $(x,f(y)-f(x))\in V(\kk)$. Note that $V$ is closed because it can be written as an intersection of $H\times \ker\mu$ with the image of the proper morphism \black \[(\pi_1,f\circ\pi_2-f\circ\pi_1) \colon H\times H\to H\times (E\times E').\] 	 
	 Let $V_i$ denote any one-dimensional component of $V$, necessarily of the form $H\times \{p\}$. Then $f(x)+p\in f(H(\kk))$ for all $x\in H(\kk)$, so by \autoref{lem:notranslates} we can conclude $p=0$. In particular, for all $x\in H(\kk)$ in some nonempty open subset of $H$ (away from the projections to $H$ of the zero-dimensional components of $V$), if for some $y\in H(\kk)$ we have $\mu(f(y))=\mu(f(x))$, then  $(x,f(y)-f(x))\in H(\kk)\times \{0\}$  and hence we must have $f(y)=f(x)$. 
	 
	 Thus the map $f(H)\to \mu(f(H))$ is injective on $\kk$-points when restricted to a nonempty open subset. As $\mu$ is \'etale, its restriction to $f(H)$ is birational onto its image. Hence the composition $\mu\circ f$  is a composition of two maps, each birational onto its image. Since
	 \[\mu\circ f\circ\iota=\mu\circ [-1]\circ f=-\mu\circ f,\]
	 $\mu\circ f$ is compatible.
	 
\end{proof}

\subsection{Proof of \autoref{infinite_curves}}\label{infcurves_proof}

Let $A/\kk$ be an abelian variety, and $E_0,E_0'/\kk$ elliptic curves with a homomorphism $\psi \colon E_0\times E_0'\to A$ having finite kernel. Replacing $A$ with the image of this homomorphism if needed, and replacing the isogeny $\mu \colon A\to B$ with an appropriate restriction of the domain and codomain, we may assume that the map $\psi \colon E_0\times E_0'\to A$ is an isogeny. 

Now fix a positive integer $n$.  Since $\kk$ has characteristic zero, there are infinitely many pairwise non-isomorphic elliptic curves $E$ (respectively $E'$) isogenous to $E_0$ (respectively $E_0'$). For each such pair $E,E'$, we can exhibit a hyperelliptic curve $H$ with genus in $[\frac16n-10,4n-2]$ and a compatible map $f \colon H\to E\times E'$ as in \autoref{lem:construction}. \black The maps $H\to E$ and $H\to E'$ have degree $2n$; since a given curve $H$ has only finitely many elliptic subcovers of bounded degree \cite[Satz 1']{tamme1972}, each isomorphism class $H$ can occur for only finitely many pairs $E,E'$. Therefore we have infinitely many pairwise non-isomorphic $H$.

By construction, the compatible maps $f \colon H\to E\times E'$ map to sections of Inose's pencil on the Kummer surface of $E\times E'$, and so $f$ remains compatible after composing with any isogeny by \autoref{prop:multcompatible}. Thus for any isogeny $\mu \colon A\to B$, the composition
\[f \colon H\to E\times E'\to E_0\times E_0'\xrightarrow{\psi} A\xrightarrow{\mu}B\]
is compatible.

\begin{rem}\label{oldconstructions}  Suppose $A=\Jac(C)$ is a Jacobian of a genus $2$ curve. An old construction due to Humbert (\cite{Humbert}) constructs an explicit curve $H/\kk$ such that $A/L\simeq \Jac(H)$, where $L$ is a Lagrangian subgroup of $A[2]$. We refer to \cite{Bost/Mestre} for a more recent reference. 
Humbert's construction uses Gauss' arithmetic-geometric mean to achieve ``doubling of the period matrix". Since every genus $2$ curve is hyperelliptic, one can iterate the above construction to produce countably many hyperelliptic curves $H_n$ (pairwise non-isomorphic, because the period matrices are distinct) such that there is an isogeny $A\xrightarrow{\phi_n}\Jac(H_n)$ with kernel of order equal to a power of $4$.
   Since every abelian surface over an algebraically closed field is isogenous  to a Jacobian, the above construction shows that we can always find countably many non-isomorphic genus $2$ hyperelliptic curves mapping to a given abelian surface $A$.
\end{rem}

\subsection{Explicit Examples}\label{example:explicit}

Pick Legendre curves $E_a,E_b$ as in \autoref{Kummerfibration}, with 
\[E_a\colon y_1^2=x_1(x_1-1)(x_1-a),\qquad E_b \colon y_2^2=x_2(x_2-1)(x_2-b).\]
If we have a section $\phi\colon\mathbb{P}^1\to K$ given by $\phi(t)=(\phi_1(t),\phi_2(t),t)$ for some rational functions $\phi_1,\phi_2\in\kk(t)$, then the curve
\[C\colon y^2=\phi_1(t)(\phi_1(t)-1)(\phi_1(t)-a)\]
fits into a commutative diagram
\[\begin{tikzcd}
	C\arrow["(t{,}y)\mapsto (\phi_1(t){,}y{,}\phi_2(t){,}ty)"]{rrrr}\arrow["(t{,}y)\mapsto t"]{d} & & & & E_a\times E_b\arrow["(x_1{,}y_1{,}x_2{,}y_2)\mapsto \left(x_1{,}y_1{,}\frac{y_2}{y_1}\right)"]{d}\\
	\mathbb{P}^1\arrow["t\mapsto (\phi_1(t){,}\phi_2(t){,}t)"]{rrrr} & & & & K,
\end{tikzcd}\]
and the normalization of $C$ will be a hyperelliptic curve with a compatible map to $E_a\times E_b$. 

In the following sections we provide explicit equations for the hyperelliptic curves obtained from various choices of sections $\phi \colon \mathbb{P}^1\to K$. \black
The sections we consider will be of the form
\[pA_{22}\oplus qA_{33}\oplus rA_{23}\oplus sA_{32}\]
as in \autoref{eq:Psection}.

\subsubsection{Degenerate sections}

There are nine quadruples $(p,q,r,s)$ with $n(p,q,r,s)=0$, namely
\begin{gather*}
	(0,0,0,0),(1,0,0,0),(0,1,0,0),(0,0,1,0),(0,0,0,1), \\
	(1,0,1,0),(1,0,0,1),(0,1,1,0),(0,1,0,1).
\end{gather*}
These correspond to the zero section $A_{11}$, the four generators $A_{22}$, $A_{33}$, $A_{23}$, and $A_{32}$, as well as the curves $A_{31}$, $A_{13}$, $A_{12}$, and $A_{21}$ by \autoref{eq:Hij_relations}. These are the nine sections of Inose's pencil that contract to points in $(E_a\times E_b)/\langle -1\rangle$.

\subsubsection{Genus $2$ construction}\label{sec:explicitgen2}

There are $18$ quadruples $(p,q,r,s)$ with $n(p,q,r,s)=1$, namely
\begin{gather*}
	(1,1,0,0),(-1,1,0,0),(1,-1,0,0),(0,0,1,1),(0,0,-1,1),(0,0,1,-1),\\
	(1,1,1,0),(-1,1,1,0),(1,-1,1,0),(1,0,1,1),(1,0,-1,1),(1,0,1,-1),\\
	(1,1,0,1),(-1,1,0,1),(1,-1,0,1),(0,1,1,1),(0,1,-1,1),(0,1,1,-1).
\end{gather*}
\black
Consider the case $(p,q,r,s)=(0,0,1,1)$. Using Mordell-Weil addition we obtain 
\[\mathcal{P}=(1,b, t )\oplus(a,1, t )=\left(\frac{(a-b)(b-1)^2}{(a-1)^3 t^2-(b-1)^3},\;\frac{(a-b)(a-1)^2 t^2}{(a-1)^3 t^2-(b-1)^3},\; t \right).\]
By the discussion above, we obtain the curve
\begin{align*}
	C \colon y^2&=\left((a-1)^3 t^2-(b-1)^3\right)\left((a-b)(b-1)^2-(a-1)^3 t^2+(b-1)^3\right)\\
	&\qquad\left((a-b)(b-1)^2-a(a-1)^3 t^2+a(b-1)^3)\right)
\end{align*}
mapping to $E_a\times E_b$. A change of variables takes this to the curve
\[H \colon y^2=\left(t^2-\frac{b-1}{a-1}\right)(t^2-1)\left(t^2-\frac{b}{a}\right),\]
which defines a smooth genus $2$ curve over $\kk$ for all $a,b\in\kk-\{0,1\}$ with $a\neq b$. This is the curve computed by Scholten in \cite[Section 7]{Scholten}, up to a change of variables.

\begin{rem}
We can similarly produce a genus $2$ curve from each of the other $17$ sections with $n(p,q,r,s)=1$, but several of these sections produce isomorphic curves. It appears that the eighteen curves lie in at most six distinct isomorphism classes. Scholten's construction, which depends on an ordering of the points of order $2$ of $E_a$ and of $E_b$,  also produces up to six distinct isomorphism classes of curves. Experiments indicate that all eighteen of the genus $2$ curves, together with their compatible maps into $E_a\times E_b$, can be recognized as instances of Scholten's construction. \black
\end{rem}
\subsubsection{Genus $6$}\label{sec:explicitgen6}

There are $45$ quadruples $(p,q,r,s)$ with $n(p,q,r,s)=2$. Consider the quadruple $(2,0,0,0)$. \black In this case we have
\[\mathcal{P}=2(1,1, t )=(\phi_1( t ),\, \phi_2( t ),\, t )\]
with 
\[\phi_1( t )=\frac{a(a-1)^3  t^4+(a-1)^2(b-1)(ab-a+b) t^2+(a-1)b(b-1)^3}{(a-1)^3 t^4-(ab-a-b-2)(a-1)^2(b-1)^2 t^2+(b-1)^3}\]
and $\phi_2(t)$ some other rational function. As described above, the curve $C \colon y^2=\phi_1(t)(\phi_1(t)-1)(\phi_1(t)-a)$ maps to $E_a\times E_b$; after a change of variables, we obtain 
\begin{align*}
	y^2&=\left(t^2-\frac{b(1 - b)}{a}\right)\left(t^2+(b-1)^2\right)\left(t^2+\frac{b-1}{a-1}\right)\\
	&\qquad\left((a - 1) t^4 - (a b - a - b - 2) (a - 1) (b - 1)^2 t^2 + (b - 1)^3\right)\\
	&\qquad\left(t^2+\frac{(b-1)^2(ab-b-1)}{a-1}\right)\left(t^2+\frac{(b-1)^2}{(a-1)(ab-a-1)}\right).
\end{align*}
For almost all values of $a,b$, the roots of this polynomial are distinct, so we have a genus $6$ hyperelliptic curve over $\kk$ mapping birationally into $E_a\times E_b$.

\subsection{Computations}\label{computations}
We end this section with some explicit computations that improve the previous results of the authors in \cite{GazakiLove2022}. Code and data supporting the results in this section are available on Github \cite{Lovecode}. \black

 In this subsection we work with varieties defined over $\Q$. 
For two elliptic curves $E,E'$ over $\Q$ there is a homomorphism 
\[E(\Q)\otimes E'(\Q)\to F^2(E\times E'),\]
induced by $p\otimes q\mapsto z_{(p,0),(0,q)}$. We label the image of this homomorphism $F^2(E\times E')_{\text{comp}}$  following \cite{GazakiLove2022}. \black Beilinson's \autoref{beilconj2} implies that $F^2(E\times E')_{\text{comp}}$ is torsion. Since $E(\Q)\otimes E'(\Q)$ is finitely generated, we in fact expect $F^2(E\times E')_{\text{comp}}$ to be finite. When $E, E'$ are non-isogenous and have both positive rank over $\Q$, proving this finiteness is a nontrivial task. 

In \cite{GazakiLove2022}, the authors developed an algorithm to prove $F^2(E\times E')_{\text{comp}}$ is finite for a large collection of pairs $(E,E')$. The primary technique was to use explicit genus $2$ curves $H$ with maps into $E\times E'$. An explicit formula for these curves is due to Scholten \cite[Section 7]{Scholten}. As discussed in \autoref{sec:explicitgen2}, these curves can also be recognized as pullbacks of sections of Inose's pencil with $n(p,q,r,s)=1$.

As described in \autoref{infcurves_proof}, there are two methods that can be used to generate new hyperelliptic curves with compatible maps: one can use more sections of the elliptic fibration, and/or one can modify the elliptic curves by isogenies first before computing the Kummer surface. Using each of these techniques, we can expand the set of pairs $(E,E')$ for which  $F^2(E\times E')_{\text{comp}}$ is provably finite.

Specifically, we considered the same set of pairs $(E,E')$ as in \cite[Section 4.2.1]{GazakiLove2022}; these are curves of positive rank with fully rational $2$-torsion. For each pair $(E,E')$, we considered all sections with $n(p,q,r,s)\leq 3$ in the elliptic fibration of the Kummer surface (associated to curves of genus $2$, $6$, and $10$); while the discussion in this paper was over $\QQ$, this construction can in fact be done over $\Q$. When either $E$ or $E'$ has a rational isogeny to another elliptic curve with fully rational $2$-torsion, we repeated the computation for the isogenous curves as well.  For all hyperelliptic curves produced, we searched for points up to the same height bound. \black \autoref{data} contains both the old data as well as the updated counts, showing that using more sections and isogenies allows for more relations to be generated, even if the height bound for rational point search is kept constant. \black

Here is a discussion of the first row. We considered the first $100$ elliptic curves of rank $1$ with fully rational $2$-torsion ordered by conductor, as listed in \cite{lmfdb}. \black There are $4950$ pairs of distinct curves in this set. Of these, the algorithm from \cite{GazakiLove2022} (using Scholten's genus $2$ curves) could prove $F^2(X)\new$ is finite for $2602$ of the pairs. By using genus $6$ and $10$ curves as well, the desired relations could be found for an additional $174$ pairs. By allowing $E$ and $E'$ to be replaced by isogenous curves before computing the Kummer surface, \black the desired relations could be found for an additional $506$ pairs.

As the table shows, we found improvement even in the cases when at least one curve has rank $>1$. Even though the total number of cases when we can verify that $F^2(E\times E')_{\text{comp}}$ is finite is still small, we did see substantial improvement on the number of new relations produced. For instance, consider non-isomorphic pairs of rank $2$ curves with fully rational $2$-torsion (second row of the table). While the old algorithm could prove finiteness of $F^2(X)\new$ for $995$ pairs, the updated algorithm proves finiteness for an additional $226$ pairs. As described in \autoref{elliptic1}, this requires producing four independent relations. There are an additional $433$ pairs for which the updated algorithm could find a larger set of independent relations than the old algorithm could, even if four independent relations could not be found. \black

\begin{rem} If $H\to E\to E'$ is any of the genus $2$ curves obtained by the elliptic fibration (cf. \eqref{sec:explicitgen2}), the Jacobian $J_H$ is isogenous to $E\times E'$. It then follows by \autoref{elliptic2} that whenever we can verify finiteness of $F^2(E\times E')\new$, we can also deduce finiteness of $F^2(J_H)\new$. Here $F^2(J_H)\new$ is the subgroup of $F^2(J_H)$ generated by $0$-cycles of the form $z_{a,b}$ with $a,b\in J_H(\Q)$. 
\end{rem}
\black

\begingroup
\renewcommand*{\arraystretch}{1.1}
\begin{table}
	\begin{tabular}{| c | c | c | c | c | c | c |} 
		\hline
		\# of $E$ & $\text{rk}\, E(\Q)$ & \# of $E'$ & $\text{rk}\, E'(\Q)$ & \begin{tabular}{@{}c@{}}total \# \\  of pairs\end{tabular} & \begin{tabular}{@{}c@{}}\# $F^2(X)\new$ \\  finite (\cite{GazakiLove2022})\end{tabular} & \begin{tabular}{@{}c@{}}\# $F^2(X)\new$ \\  finite (updated)\end{tabular} \\
		\hline
		100 & 1 & 100 & 1& 4950 & 2602 & 3282 \\
		100 & 2 & 100 & 2& 4950 & 995 & 1221\\
		100 & 2 & $=E$ & & 100 & 70 & 75 \\
		20 & 3 & 20 & 3& 190 & 17 & 17\\
		20 & 3 & $=E$ &   & 20& 8 & 9 \\
		100 & 1 & 100 & 2& 10000 & 3311 & 4350\\
		500 & 1 & 20 & 3 & 10000 & 955 & 1330\\
		500 & 2 & 20 & 3 & 10000 & 615 & 756 \\
		\hline
	\end{tabular}
	\caption{Counting pairs of elliptic curves over $\Q$ with $F^2(X)\new$ finite for $X=E\times E'$. All data except for the last column is copied from \cite{GazakiLove2022}.}\label{data}
\end{table}
\endgroup

\black
\vspace{3pt}

\section{Higher dimensions}\label{abvars}
In this section we explore to what extent the results of \autoref{abelsurfacesection} can be extended to higher dimensional abelian varieties. We saw that for an abelian surface $A$ over $\QQ$ Beilinson's conjecture is equivalent to showing $z_{a,a}=0$ for all $a\in A$ (see \autoref{zaa}). Our goal is to obtain a result of similar flavor for an abelian variety $A/\QQ$ of dimension $d$. 
For that we need to recall certain facts about motivic filtrations of the group $\CH_0(A)$. 
 \subsection{The Pontryagin Filtration} Let $A$ be an abelian variety of dimension $d$ over an algebraically closed field $\kk$. The group law of $A$ makes $\CH_0(A)$ into a group ring by defining for $a,b\in A(\kk)$ the Pontryagin product 
 \[[a]\odot[b]:=[a+b].\] We denote by $G^1(A)$ the augmentation ideal of $\CH_0(A)$ and for $r\geq 1$, $G^r(A):=(G_1(A))^r$ its $r$-th power. We also set $G^0(A)=\CH_0(A)$. A straightforward computation shows that $G^1(A)$ is precisely the subgroup $F^1(A)$ of zero-cycles of degree $0$ and $G^2(A)$ coincides with the kernel $F^2(A)$ of the Albanese map. The filtration $\{G^r(A)\}_{r\geq 0}$ is known as the \textit{Pontryagin filtration} of $\CH_0(A)$ and it has been previously studied by Beauville and Bloch (\cite{Beauville1983, Beauville1986, Bloch1976}). \black 
 These groups have the following explicit generators. 
 \begin{eqnarray*}
&&G^{1}(A)=\langle[a]-[0]:a\in A(\kk)\rangle=\ker(\deg),\\
&&G^{2}(A)=\langle[a+b]-[a]-[b]+[0]:a,\;b\in A(\kk)\rangle,\\
&&G^{3}(A)=\langle[a+b+c]-[a+b]-[b+c]-[a+c]+[a]+[b]+[c]-[0]:a,\;b,\;c\in A(\kk)\rangle,\\
&&\ldots \\
&&G^{r}(A)=\left\langle\sum_{j=0}^{r}(-1)^{r-j}\sum_{1\leq\nu_{1}<\dots<\nu_{j}\leq r}[a_{\nu_{1}}+\dots+a_{\nu_{j}}]:a_{1},\dots,a_{r}\in A(\kk)\right\rangle. 
\end{eqnarray*} 
\begin{notn} We will denote by $z_{a_1,\ldots,a_r}$ the generator of $G^r(A)$ corresponding to points $a_1,\ldots,a_r\in A(\kk)$. Notice that $z_{a_1,\ldots,a_r}=([a_1]-[0])\odot\cdots\odot([a_r]-[0]).$
\end{notn}  Beauville and Bloch both showed a vanishing $G^{d+1}(A)=0$. Moreover, it follows by the main theorem of \cite{Beauville1986} and its proof that we have a 
 rational decomposition \black
 \[\CH_0(A)\otimes\Q\simeq \bigoplus_{r=0}^d\frac{G^r(A)\otimes\Q}{G^{r+1}(A)\otimes\Q}.\]

 \subsection{The Gazaki filtration}\label{gazakifil} The first author described in \cite{Gazaki2015} the graded quotients
 \[\displaystyle\frac{G^r(A)\otimes\Q}{G^{r+1}(A)\otimes\Q}\] \black
  as symmetric products of the abelian variety $A$. Namely, she defined a second integral filtration $\{F^r(A)\}_{r\geq 0}$ of $\CH_0(A)$ with the following properties. 
 \begin{enumerate}
 \item[(a)] $G^r(A)\subseteq F^r(A)$ for every $r\geq 0$ (see \cite[Proposition 3.3]{Gazaki2015}). 
 \item[(b)] For every $r\geq 1$, 
 \[\displaystyle G^r(A)\otimes\Z\left[\frac{1}{(r-1)!}\right]=F^r(A)\otimes\Z\left[\frac{1}{(r-1)!}\right]\]
 \black (see \cite[Proposition 4.1]{Gazaki2015}). In particular, $F^1(A)=G^1(A)$ and $F^2(A)=G^2(A)$. 
 \item[(c)] For every $r\geq 0$ there is a homomorphism \[\Psi_r \colon  \overbrace{A(\kk)\otimes\cdots\otimes A(\kk)}^r\to F^r(A)/F^{r+1}(A), \;\;a_1\otimes\cdots\otimes a_r\mapsto z_{a_1,\ldots,a_r}\] (see \cite[Theorem 1.3]{Gazaki2015}). The map $\Psi_r$ factors through a certain symmetric Somekawa $K$-group $S_r(\kk;A)$ which we review below.  
 \end{enumerate} 
\begin{defn} Let $A$ be an abelian variety over $\kk$. For $r\geq 1$ the Somekawa $K$-group $K_r(\kk;A):=K(\kk;\overbrace{A,\ldots,A}^r)$ attached to $r$ copies of $A$ is the quotient
\[K_r(\kk;A)=\frac{\overbrace{A(\kk)\otimes\cdots\otimes A(\kk)}^r}{(\textbf{WR})},\] where $(\textbf{WR})$ is the subgroup generated by the following type of elements. Let $C$ be a smooth projective curve over $\kk$ that admits regular maps $g_i \colon C\to A$ for $i=1,\ldots,r$. Let $f\in \kk(C)^\times$. Then we require 
\begin{equation}\label{eq:WR}
	\sum_{x\in C}\ord_x(f)g_1(x)\otimes\cdots\otimes g_r(x)\in(\textbf{WR}).
\end{equation}
 The symmetric $K$-group $S_r(\kk;A)$ is the quotient of $K_r(\kk;A)$ by the action of the symmetric group in $r$ variables. 
\end{defn}
The relation $(\textbf{WR})$ is known as \textit{Weil Reciprocity}. We will denote the generator of $S_r(\kk;A)$ corresponding to the tensor $a_1\otimes\cdots\otimes a_r$ as a symbol $\{a_1,\ldots,a_r\}$. The following proposition is the analog of \autoref{zaa} for higher dimensions and it is an easy consequence of the above analysis. 
\begin{prop}\label{higherdim} Let $A$ be an abelian variety over an algebraically closed field $\kk$. If the symmetric $K$-group $S_2(\kk;A)$ is torsion, then $F^2(A)=0$. In particular, if $\{a,a\}$ is torsion for every $a\in A$, then $F^2(A)=0$. 
\end{prop}

\begin{proof} 
First we claim that if $S_2(\kk;A)$ is torsion, then $S_r(\kk;A)$ is torsion for all $r\geq 2$. This follows by a product formula for Somekawa $K$-groups proved by the authors in \cite[Proposition 3.1]{GazakiLove2022}. This proposition gives a homomorphism
  \[\rho \colon K_2(\kk;A)\otimes \overbrace{A(\kk)\otimes\cdots\otimes A(\kk)}^{r-2}\rightarrow K_r(\kk;A)\] given by concatenation of symbols, $\rho(\{a_1,a_2\}\otimes a_3\otimes\cdots\otimes a_n)=\{a_1,\ldots, a_n\}$. We see that $\rho$ is surjective by construction. The homomorphism $\rho$ clearly induces a surjective homomorphism
  \[S_2(\kk;A)\otimes A(\kk)\otimes\cdots\otimes A(\kk)\rightarrow S_r(\kk;A),\] and hence if $S_2(\kk;A)$ is torsion, then so is $S_r(\kk;A)$ for all $r\geq 3$. 
  
  From now on suppose that $S_r(\kk;A)$ is torsion for all $r\geq 2$. It follows by \eqref{gazakifil} (c) that there is a homomorphism $\Psi_d \colon  S_d(\kk;A)\to F^d(A)/F^{d+1}(A)$ sending the symbol 
  $\{a_1,\ldots,a_d\}$ to $z_{a_1,\ldots,a_d}$. \black Moreover, it follows by (b) that $\displaystyle F^{d+1}(A)\otimes\Z\left[\frac{1}{d!}\right]=0$, since $G^{d+1}(A)=0$. Since $F^{d+1}(A)$ is a subgroup of the Albanese kernel, which is torsion-free, it follows that $F^{d+1}(A)=0$, and hence we have a well-defined homomorphism 
$\Psi_d \colon S_d(\kk;A)\to F^d(A).$ Since $S_d(\kk;A)$ is torsion and $F^d(A)$ is torsion-free, it follows that $\Psi_d=0$. Notice that by construction the map $\Psi_d$ surjects onto $G^d(A)\subseteq F^d(A)$. Thus, we conclude that $G^d(A)=0$, and by torsion-freeness also $F^d(A)=0$. We now proceed by reverse induction. Having proved that $G^r(A)=F^r(A)=0$ for some $r\geq 3$, we get that $G^{r-1}(A)=F^{r-1}(A)=0$ by considering the homomorphism
\[\Psi_{r-1} \colon  S_{r-1}(A)\rightarrow F^{r-1}(A)\] which surjects onto $G^{r-1}(A)$. 

The last claim follows by symmetry and bilinearity. The group $S_2(\kk;A)$ is torsion if an only if $\{a,b\}$ is torsion for all $a,b\in A$. But $2\{a,b\}=\{a,a\}+\{b,b\}-\{a+b,a+b\}$. 

 \end{proof}
 
The following is the analog of \autoref{odd} and the proof is along the same lines. 
\begin{prop}\label{odd2} Let $A$ be an abelian variety of dimension $d$ over $\kk$ and let $a\in A(\kk)$ be a hyperelliptic point (in the sense of \autoref{goodpoint}). Then $\{a,a\}\in S_2(\kk;A)$ is torsion. 
\end{prop}
\begin{proof}
By assumption there exists some nonzero multiple of $a$ that lies in the image of a morphism $f \colon H\to A$ from a hyperelliptic curve such that negation on $A$ induces the hyperelliptic involution on $A$. \black 
As usual, using bilinearity we may assume that $a=f(q)$ is itself in the image of $f$. 
Then $[q]+[\iota(q)]-2[w]$ is a principal divisor on $H$, where $w$ is a Weierstrass point of $H$. Let $h\in \kk(C)^\times$ be the function whose divisor is $[q]+[\iota(q)]-2[w]$. We apply the Weil Reciprocity relation $(\textbf{WR})$ of $S_2(\kk;A)$ for the following choices: $g_1=g_2=f$ and $h\in\kk(C)^\times$ as above. Given that $f(q)=-f(\iota(q))$, it follows
\[\{a,a\}+\{-a,-a\}=0\Rightarrow 2\{a,a\}=0.\] 

\end{proof}

Because of \autoref{odd2}, \autoref{beilreduce1} generalizes verbatim to higher dimensions. 

\begin{cor}\label{beilreducehigher} 
 Let $A$ be an abelian variety of dimension $d\geq 2$ over a number field $k$. Suppose that for every finite extension $L/k$, if we let $r$ be the Mordell-Weil rank of $A(L)$ and $a_1,\ldots,a_r\in A(L)$ independent elements in $A(L)$, there exists $M\subseteq \Z^r$ such that
	\begin{itemize}
		\item the set of tensors $\mathbf{m}\otimes \mathbf{m}$ for $\mathbf{m}\in M$ spans $\Sym^2\Q^r$ as a $\Q$-vector space, and
		\item for each $(m_1,\ldots,m_r)\in M$, the point $m_1a_1+\cdots+m_ra_r$ is a hyperelliptic point.
	\end{itemize} 
	Then Beilinson's \autoref{beilconj2} is true for $A_{\kk}$. \black

\end{cor}

\begin{rem}  
	The reason we chose to state most of our results for abelian surfaces is because it seems unlikely that for a general abelian variety $A/\QQ$ of $\dim(A)\geq 3$ relations arising from hyperelliptic curves will be enough to prove Beilinson's conjecture. This is because in $\dim\geq 3$ the quotient $A/\langle -1\rangle$ by negation is in general not Calabi-Yau, and hence it does not contain rational curves (see \cite[Remark 7.7]{bogomolov_tschinkel}), which is our main method to construct hyperelliptic curves mapping to $A$. However, 
	 there may be a large collection of hyperelliptic curves in $A$ \black for special classes of abelian varieties. For example, when $J$ is the Jacobian of a hyperelliptic curve of genus $g\geq 2$, Bogomolov and Tschinkel proved the analog of Bogomolov's conjecture for $J/\overline{\F}_p$ (\cite[Theorem 7.1]{bogomolov_tschinkel})).  
\end{rem}

\vspace{5pt}

		\bibliographystyle{amsalpha}
		\bibliography{bibfile,bibabelian}

\providecommand{\bysame}{\leavevmode\hbox to3em{\hrulefill}\thinspace}
\providecommand{\MR}{\relax\ifhmode\unskip\space\fi MR }
\providecommand{\MRhref}[2]{%
  \href{http://www.ams.org/mathscinet-getitem?mr=#1}{#2}
}
\providecommand{\href}[2]{#2}
\begin{thebibliography}{{LMF}22}

\bibitem[BB91]{borweinborwein}
J.~M. Borwein and P.~B. Borwein, \emph{A cubic counterpart of {J}acobi's
  identity and the {AGM}}, Trans. Amer. Math. Soc. \textbf{323} (1991), no.~2,
  691--701. \MR{1010408}

\bibitem[Bea83]{Beauville1983}
A.~Beauville, \emph{Quelques remarques sur la transformation de {F}ourier dans
  l'anneau de {C}how d'une vari\'{e}t\'{e} ab\'{e}lienne}, Algebraic geometry
  ({T}okyo/{K}yoto, 1982), Lecture Notes in Math., vol. 1016, Springer, Berlin,
  1983, pp.~238--260. \MR{726428}

\bibitem[Bea86]{Beauville1986}
\bysame, \emph{Sur l'anneau de {C}how d'une vari\'et\'e ab\'elienne}, Math.
  Ann. \textbf{273} (1986), no.~4, 647--651. \MR{826463 (87g:14049)}

\bibitem[Bei84]{Beilinson1984}
A.~A. Beilinson, \emph{Higher regulators and values of {$L$}-functions},
  Current problems in mathematics, {V}ol. 24, Itogi Nauki i Tekhniki, Akad.
  Nauk SSSR, Vsesoyuz. Inst. Nauchn. i Tekhn. Inform., Moscow, 1984,
  pp.~181--238. \MR{760999}

\bibitem[Blo75]{Bloch1975}
Spencer Bloch, \emph{{$K_{2}$} of {A}rtinian {$Q$}-algebras, with application
  to algebraic cycles}, Comm. Algebra \textbf{3} (1975), 405--428. \MR{0371891}

\bibitem[Blo76]{Bloch1976}
\bysame, \emph{Some elementary theorems about algebraic cycles on {A}belian
  varieties}, Invent. Math. \textbf{37} (1976), no.~3, 215--228. \MR{0429883
  (55 \#2892)}

\bibitem[BM88]{Bost/Mestre}
Jean-Beno\^{i}t Bost and Jean-Fran\c{c}ois Mestre, \emph{Moyenne
  arithm\'{e}tico-g\'{e}om\'{e}trique et p\'{e}riodes des courbes de genre
  {$1$} et {$2$}}, Gaz. Math. (1988), no.~38, 36--64. \MR{970659}

\bibitem[BT05]{bogomolov_tschinkel}
Fedor Bogomolov and Yuri Tschinkel, \emph{Rational curves and points on {$K3$}
  surfaces}, Amer. J. Math. \textbf{127} (2005), no.~4, 825--835. \MR{2154371}

\bibitem[BV04]{Beauville/Voisin2004}
Arnaud Beauville and Claire Voisin, \emph{On the {C}how ring of a {$K3$}
  surface}, J. Algebraic Geom. \textbf{13} (2004), no.~3, 417--426.
  \MR{2047674}

\bibitem[Gaz15]{Gazaki2015}
Evangelia Gazaki, \emph{On a filtration of {$CH_0$} for an abelian variety},
  Compos. Math. \textbf{151} (2015), no.~3, 435--460. \MR{3320568}

\bibitem[GL24]{GazakiLove2022}
Evangelia Gazaki and Jonathan Love, \emph{Torsion phenomena for zero-cycles on
  a product of curves over a number field}, Res. Number Theory \textbf{10}
  (2024), no.~2, Paper No. 35, 20. \MR{4724170}

\bibitem[Hum01]{Humbert}
G.~Humbert, \emph{Sur la transformation ordinaire de fonctions ab\'eliennes},
  J. Math. Pures Appl. (1901), no.~7, 395--417.

\bibitem[KS08]{kuwatashioda}
Masato Kuwata and Tetsuji Shioda, \emph{Elliptic parameters and defining
  equations for elliptic fibrations on a {K}ummer surface}, Algebraic geometry
  in {E}ast {A}sia---{H}anoi 2005, Adv. Stud. Pure Math., vol.~50, Math. Soc.
  Japan, Tokyo, 2008, pp.~177--215. \MR{2409557}

\bibitem[{LMF}22]{lmfdb}
The {LMFDB Collaboration}, \emph{The {L}-functions and modular forms database},
  \url{http://www.lmfdb.org}, 2022, [Online; accessed 16 March 2022].

\bibitem[Lov25]{Lovecode}
Jonathan Love, \emph{Rational equivalences of zero-cycles},
  \url{https://github.com/jonathanrlove/zero-cycles/}, 2025, [Online; accessed
  2 May 2025].

\bibitem[Mum68]{Mumford1968}
D.~Mumford, \emph{Rational equivalence of {$0$}-cycles on surfaces}, J. Math.
  Kyoto Univ. \textbf{9} (1968), 195--204. \MR{249428}

\bibitem[Roj80]{Rojtman1980}
A.~A. Rojtman, \emph{The torsion of the group of {$0$}-cycles modulo rational
  equivalence}, Ann. of Math. (2) \textbf{111} (1980), no.~3, 553--569.
  \MR{577137}

\bibitem[Sch18]{Scholten}
Jasper Scholten, \emph{Genus 2 curves with given split jacobian}, Cryptology
  ePrint Archive, Report 2018/1137, 2018, \url{https://ia.cr/2018/1137}.

\bibitem[Sha03]{shaska}
Tanush Shaska, \emph{Determining the automorphism group of a hyperelliptic
  curve}, Proceedings of the 2003 {I}nternational {S}ymposium on {S}ymbolic and
  {A}lgebraic {C}omputation, ACM, New York, 2003, pp.~248--254. \MR{2035219}

\bibitem[Shi90]{shioda90}
Tetsuji Shioda, \emph{On the {M}ordell-{W}eil lattices}, Comment. Math. Univ.
  St. Paul. \textbf{39} (1990), no.~2, 211--240. \MR{1081832}

\bibitem[Shi07]{shioda07}
\bysame, \emph{Correspondence of elliptic curves and {M}ordell-{W}eil lattices
  of certain elliptic {$K3$}'s}, Algebraic cycles and motives. {V}ol. 2, London
  Math. Soc. Lecture Note Ser., vol. 344, Cambridge Univ. Press, Cambridge,
  2007, pp.~319--339. \MR{2385296}

\bibitem[Tam72]{tamme1972}
G\"{u}nter Tamme, \emph{Teilk\"{o}rper h\"{o}heren {G}eschlechts eines
  algebraischen {F}unktionenk\"{o}rpers}, Arch. Math. (Basel) \textbf{23}
  (1972), 257--259. \MR{311666}

\bibitem[UU21]{ulmerurzua}
Douglas Ulmer and Giancarlo Urz\'{u}a, \emph{Bounding tangencies of sections on
  elliptic surfaces}, Int. Math. Res. Not. IMRN (2021), no.~6, 4768--4802.
  \MR{4230412}

\bibitem[UU22]{ulmerurzua-transversality}
\bysame, \emph{Transversality of sections on elliptic surfaces with
  applications to elliptic divisibility sequences and geography of surfaces},
  Selecta Math. (N.S.) \textbf{28} (2022), no.~2, Paper No. 25, 36.
  \MR{4357480}

\bibitem[Wan16]{wang2016}
Liuquan Wang, \emph{Explicit formulas for partition pairs and triples with
  3-cores}, J. Math. Anal. Appl. \textbf{434} (2016), no.~2, 1053--1064.
  \MR{3415707}

\end{thebibliography}
		\vspace{10pt}
		\appendix

		\section{Intersection number computations}\label{appendix:intersection}
		
		In this appendix we provide the proofs of \autoref{lem:intersection_formula}, \autoref{lem:projection_degree}, and \autoref{bound_transverse}.
		
		As in \autoref{Kummerfibration}, let $E_a,E_b/\kk$ be elliptic curves in Legendre form, with $p_0,\ldots,p_3\in E_a[2]$ and $q_0,\ldots,q_3\in E_b[2]$.
		We describe a certain collection of rational curves in the Kummer surface $K$ of $E_a\times E_b$, following the notation of 
		\begin{itemize}
			\item For $0\leq i\leq 3$, $p_i\times E_b$ maps to a rational curve in $(E_a\times E_b)/\langle -1\rangle$; its pullback in $K$ is denoted $F_i$.
			\item For $0\leq j\leq 3$, $E_a\times q_j$ maps to a rational curve in $(E_a\times E_b)/\langle -1\rangle$; its pullback in $K$ is denoted $G_j$.
			\item For $0\leq i,j\leq 3$, $(p_i,q_j)$ maps to a singularity in $(E_a\times E_b)/\langle -1\rangle$; its blowup in $K$ is denoted $A_{ij}$.
		\end{itemize}
		The divisor
		\[D:=\sum_{0\leq i,j\leq 3} A_{ij}\]
		is the exceptional locus of the map $K\to A/\langle -1\rangle$.
		If we remove the curves 
		\begin{align}\label{eq:inf_fiber}
			F_1,F_2,F_3,G_0,A_{00},A_{10},A_{20},A_{30}
		\end{align}
		(corresponding to points $(x_1,y_1,x_2,y_2)\in E_a\times E_b$ with $y_1=0$ or $y_2=\infty$)
		we obtain an affine chart of $K$ defined by the equation
		\begin{equation}\label{affinepatch}
			x_1(x_1-a)(x_1-a)t^2=x_2(x_2-1)(x_2-b),
		\end{equation}
		with the rational map $E_a\times E_b\dashrightarrow K$ given by 
		\[(x_1,y_1,x_2,y_2)\mapsto \left(x_1,x_2,\frac{y_2}{y_1}\right).\]
		We have a map $\xi \colon K\to\mathbb{P}^1$ defined by $(x_1,x_2,t)\mapsto t$; this determines a fibration of $K$ called \emph{Inose's pencil}.  	The fiber at $t=\infty$ is
		\[3G_0+2(A_{10}+A_{20}+A_{30})+F_1+F_2+F_3,\]
		of Kodaira type $IV^*$. The fiber at $t=0$ is
		\[3F_0+2(A_{01}+A_{02}+A_{03})+G_1+G_2+G_3,\]
		also of Kodaira type $IV^*$. There may be other singular fibers (depending on $a$ and $b$), but if $E_a$ and $E_b$ have distinct $j$-invariants, then the fibers at $t=0,\infty$ are the only reducible fibers~\cite[Proposition 5.1]{shioda07}. The curves $A_{ij}$ for $1\leq i,j\leq 3$ are sections of the fibration, and the curve $A_{00}$ is a multisection, intersecting every fiber with multiplicity $3$. 
		
		If we fix $A_{11}$ as the zero section, $K$ obtains the structure of an elliptic fibration. We use~$\oplus$ to denote Mordell-Weil addition, to distinguish it from the sum of curves as divisors. The remaining eight sections satisfy the relations~\cite[Section 6]{Scholten}
		\begin{align}\label{eq:Hij_relations}
			\begin{split}
				A_{13}&=A_{22}\oplus A_{32},\\
				A_{12}&=A_{23}\oplus A_{33},\\
				A_{21}&=A_{32}\oplus A_{33},\\
				A_{31}&=A_{22}\oplus A_{23},
			\end{split}
		\end{align}
		while $A_{22},A_{32},A_{23},A_{33}$ are independent provided the $j$-invariants of $E_a$ and $E_b$ are nonzero and distinct.
		
		For integers $p,q,r,s$, we set
		\[\mathcal{P}=pA_{22}\oplus qA_{33}\oplus rA_{23}\oplus A_{32}\]
		and
		\[n(p,q,r,s):=p(p-1)+q(q-1)+r(r-1)+s(s-1)+pq+rs.\]

		\begin{proof}[Proof of \autoref{lem:intersection_formula}]
			First we determine the intersection of $\mathcal{P}$ with the fibers at $t=0,\infty$. As a section of the fibration, $\mathcal{P}$ intersects these fibers each with multiplicity $1$, so it cannot intersect any components that occur with multiplicity greater than $1$. Hence $\mathcal{P}$ has zero intersection with each of $F_0,G_0,A_{01},A_{02},A_{03},A_{10},A_{20},A_{30}$. So $\mathcal{P}$ intersects exactly one of $F_1,F_2,F_3$, and exactly one of $G_1,G_2,G_3$. 
			
			Now for each $0\leq i\leq 3$, the divisors $D_i :=2F_i+\sum_j A_{ij}$ are fibers of the map $K\to \mathbb{P}^1$ given by $(x_1,x_2,t)\mapsto x_1$~\cite[Equation (3.3)]{shioda07}, and hence they have equal intersection product with $\mathcal{P}$. Therefore $3\mathcal{P}\cdot D_0=\mathcal{P}\cdot (D_1+D_2+D_3)$, which simplifies to
			\[3\mathcal{P}\cdot A_{00}=2\mathcal{P}\cdot (F_1+F_2+F_3)+\sum_{1\leq i,j\leq 3} \mathcal{P}\cdot A_{ij}.\]
			Since $\mathcal{P}\cdot (F_1+F_2+F_3)=1$ we obtain
			\begin{align}\label{eq:PD}
				\mathcal{P}\cdot D=\frac 23+\frac43\sum_{1\leq i,j\leq 3} \mathcal{P}\cdot A_{ij}.
			\end{align}
			This expresses $\mathcal{P}\cdot D$ entirely in terms of intersection products of sections of Inose's pencil, so we can now apply a formula relating intersection multiplicity of sections to the N\'eron-Tate pairing~\cite[Theorem 8.6]{shioda90}. We obtain
			\begin{align*}
				\sum_{1\leq i,j\leq 3} \mathcal{P}\cdot A_{ij}&=9\chi+9(\mathcal{P}\cdot A_{11})\\
				&\quad+\sum_{1\leq i,j\leq 3} \left((A_{ij}\cdot A_{11})-\langle \mathcal{P}, A_{ij}\rangle-\text{contr}_0(\mathcal{P},A_{ij})-\text{contr}_\infty(\mathcal{P},A_{ij})\right),
			\end{align*}
			where $\chi$ is the arithmetic genus of $K$, and $\text{contr}_v$ is a function on pairs of sections that depends only on which components of the fiber at $v$ they each intersect. Since $K$ is a K3 surface we have $\chi=2$ and $A_{11}\cdot A_{11}=-\chi=-2$, as well as $A_{ij}\cdot A_{11}=0$ for all other $i,j$, so this simplifies to 
			\begin{align}\label{eq:intersection_sum}
				\sum_{1\leq i,j\leq 3} \mathcal{P}\cdot A_{ij}=&16+9(\mathcal{P}\cdot A_{11})-\sum_{1\leq i,j\leq 3} \left(\langle \mathcal{P}, A_{ij}\rangle+\text{contr}_0(\mathcal{P},A_{ij})+\text{contr}_\infty(\mathcal{P},A_{ij})\right).
			\end{align}
			To compute the $\text{contr}_\infty$ terms, we note that Mordell-Weil addition induces a $\Z/3\Z$ structure on the multiplicity $1$ components $F_1,F_2,F_3$ of the fiber at $\infty$ (with $F_1$ being the identity). If $\mathcal{P}$ intersects $F_k$ then
			\begin{align*}
				\text{contr}_\infty(\mathcal{P},A_{ij})&=\left\{\begin{array}{ll}
					0,& \text{if }i=1\text{ or }k=1,\\
					2/3, & \text{if }i,k\neq 1\text{ and }i\neq k,\\
					4/3, & \text{if }i,k\neq 1\text{ and }i=k.
				\end{array}\right.
			\end{align*}
			\cite[Equation (8.16)]{shioda90}. We can therefore conclude that
			\[\sum_{1\leq i,j\leq 3}\text{contr}_\infty(\mathcal{P},A_{ij})=\left\{\begin{array}{ll}
				0,& \text{if }\mathcal{P}\text{ intersects }F_1,\\
				6, & \text{otherwise.}
			\end{array}\right.\]
			Similarly, 
			\[\sum_{1\leq i,j\leq 3}\text{contr}_0(\mathcal{P},A_{ij})=\left\{\begin{array}{ll}
				0,& \text{if }\mathcal{P}\text{ intersects }G_1,\\
				6, & \text{otherwise.}
			\end{array}\right.\]
			
			To compute $\mathcal{P}\cdot A_{11}$, we again use the relation between intersection and N\'eron-Tate pairings:
			\begin{equation}\label{PH11}
				\mathcal{P}\cdot A_{11}=\frac12\langle\mathcal{P},\mathcal{P}\rangle-2+\frac12\text{contr}_0(\mathcal{P},\mathcal{P})+\frac12\text{contr}_\infty(\mathcal{P},\mathcal{P}),
			\end{equation}
			with $\text{contr}_\infty(\mathcal{P},\mathcal{P})=0$ if~$\mathcal{P}$ intersects $F_1$ and $\frac43$ otherwise, and likewise $\text{contr}_0(\mathcal{P},\mathcal{P})=0$ if $\mathcal{P}$ intersects $G_1$ and $\frac43$ otherwise. Plugging this into \autoref{eq:intersection_sum} and noting the wonderful coincidence $9\cdot\frac12\cdot\frac43=6$, the $\text{contr}_v$ terms cancel, so the equation simplifies to
			\begin{align}
				\sum_{1\leq i,j\leq 3} \mathcal{P}\cdot A_{ij}=&\frac92\langle\mathcal{P},\mathcal{P}\rangle-2-\sum_{1\leq i,j\leq 3} \langle \mathcal{P}, A_{ij}\rangle.
			\end{align}
			Noting that
			\[\bigoplus_{1\leq i,j\leq 3} A_{ij}=3A_{22}\oplus 3A_{23}\oplus 3A_{32}\oplus 3A_{33}\]
			by \autoref{eq:Hij_relations}, and using the canonical height computations from \cite[Section 6]{Scholten}, we can continue from \autoref{eq:PD} to obtain
			\begin{align*}
				\mathcal{P}\cdot D&=\frac23 +\frac43\sum_{1\leq i,j\leq 3} \mathcal{P}\cdot A_{ij}\\
				&=\frac23 +\frac43\left(-2+\frac92\langle\mathcal{P},\mathcal{P}\rangle-\langle \mathcal{P},3A_{22}\oplus 3A_{23}\oplus 3A_{32}\oplus 3A_{33}\rangle \right)\\
				&=-2+\frac43\begin{pmatrix}
					p & q & r & s
				\end{pmatrix}\begin{pmatrix}
					4/3 & 2/3 & 0 & 0\\
					2/3 & 4/3 & 0 & 0\\
					0 & 0 & 4/3 & 2/3\\
					0 & 0 & 2/3 & 4/3
				\end{pmatrix}\left(\frac92 \begin{pmatrix}
					p \\ q \\ r \\ s
				\end{pmatrix}-\begin{pmatrix}
					3 \\ 3 \\ 3 \\ 3
				\end{pmatrix}\right)\\
				&=8n(p,q,r,s)-2.
			\end{align*}
			
		\end{proof}
	
	\begin{proof}[Proof of \autoref{lem:projection_degree}]
		 Without loss of generality we consider the map $K\to\mathbb{P}^1$ given by $(x_1,x_2,t)\mapsto x_1$.  As in the proof of \autoref{lem:intersection_formula}, we observe that $D_i=2F_i+\sum_j A_{ij}$ is a fiber of this map for $i=1,2,3$, and therefore they each have the same intersection product with~$\mathcal{P}$.  Using $\mathcal{P}\cdot (F_1+F_2+F_3)=1$ and \autoref{eq:PD}, we obtain 
		\begin{align*}
			\mathcal{P}\cdot D_1&=\frac13\mathcal{P}\cdot (D_1+D_2+D_3)\\
			&=\frac23+\frac13\sum_{1\leq i,j\leq 3} \mathcal{P}\cdot A_{ij}\\
			& =\frac23+\frac14\left(\mathcal{P}\cdot D-\frac23\right)\\
			& =2n(p,q,r,s). 
		\end{align*}
		
	\end{proof}

	\begin{proof}[Proof of \autoref{bound_transverse}]
		We first compute
		\begin{align*}
			\langle \mathcal{P},\mathcal{P}\rangle&=\begin{pmatrix}
				p & q & r & s
			\end{pmatrix}\begin{pmatrix}
				4/3 & 2/3 & 0 & 0\\
				2/3 & 4/3 & 0 & 0\\
				0 & 0 & 4/3 & 2/3\\
				0 & 0 & 2/3 & 4/3
			\end{pmatrix}\begin{pmatrix}
				p \\ q \\ r \\ s
			\end{pmatrix}\\
			&= \frac43 (p^2 + p q + q^2 + r^2 + r s + s^2),
		\end{align*}
		and note that this is nonzero for all $(p,q,r,s)\neq (0,0,0,0)$; therefore $\mathcal{P}$ is a non-torsion section.	
		Let $m_t(\mathcal{P},A_{11})$ denote the local intersection multiplicity between the sections $\mathcal{P}$ and $A_{11}$ at $t$, so that
		\[\mathcal{P}\cdot A_{11}=\sum_{t\in \kk}m_t(\mathcal{P},A_{11}).\]
		Let $I(\mathcal{P},t)$ denote the intersection multiplicity of $\mathcal{P}$ with the \emph{Betti foliation} of $K$, as defined in \cite[Section 4.1]{ulmerurzua}. The identity section $A_{11}$ is a leaf of this foliation, and so $I(\mathcal{P},t)\geq m_t(\mathcal{P},A_{11})$ for all intersection points $t$. Consequently, we have
		\begin{align*}
			\mathcal{P}\cdot A_{11}-|\{t\in\kk:m_t(\mathcal{P},A_{11})=1\}|&=\sum_{\substack{t\in \kk \\ m_t(\mathcal{P},A_{11})\geq 2}}m_t(\mathcal{P},A_{11})\\
			&\leq 2\sum_{m_t(\mathcal{P},A_{11})\geq 2}(m_t(\mathcal{P},A_{11})-1)\\
			&\leq 2\sum_{I(\mathcal{P},t)\geq 2}(I(\mathcal{P},t)-1)\\
			&\leq 2\delta+2\sum_{t\in\kk}(I(\mathcal{P},t)-1),
		\end{align*}
		where $\delta$ is the number of singular fibers of $K\to\mathbb{P}^1$; these are the only fibers for which it is possible to have $I(\mathcal{P},t)=0$. By \cite[Theorem 7.1]{ulmerurzua}, we have
		\[\sum_{t\in\kk}(I(\mathcal{P},t)-1)\leq -2\]
		(note that $d\geq 0$ because $\omega^{\otimes 12}$ has a section, see \cite[Page 1]{ulmerurzua}).
		By \autoref{PH11} we have
		\begin{align*}
			\mathcal{P}\cdot A_{11}&=\frac12\langle \mathcal{P},\mathcal{P}\rangle -2+\frac12\text{contr}_0(\mathcal{P},\mathcal{P})+\frac12\text{contr}_0(\mathcal{P},\mathcal{P})\\
			&\geq \frac12\langle \mathcal{P},\mathcal{P}\rangle-2
		\end{align*}
		since the $\text{contr}_v$ terms are non-negative. Finally, we have $\delta\leq 10$ by \cite[Proposition 5.1]{shioda07}, so combining the above results we obtain
		\begin{align*}
			|\{t\in\kk:m_t(\mathcal{P},A_{11})=1\}|&\geq  
			\mathcal{P}\cdot A_{11}-2\delta-2\sum_{t\in\kk}(I(\mathcal{P},t)-1)\\
			&\geq  
			\frac12\langle \mathcal{P},\mathcal{P}\rangle-2-20+4\\
			 &=\frac23(p^2+pq+q^2+r^2+rs+s^2)-18.
		\end{align*}
		Since $p^2+pq+q^2\geq -p-q$ for all $p,q\in\Z$ (which can be checked by noting that $p^2+pq+q^2+p+q$ takes on a minimum value of $-\frac13$ for $p,q\in\mathbb{R}$), and a similar inequality holds for $r,s$, this is greater than or equal to 
		\[\frac13(p(p-1)+q(q-1)+r(r-1)+s(s-1)+pq+rs)-18\]
		as desired. 
		
	\end{proof} 

\end{document}